\documentclass[a4paper,12pt]{article} 
\usepackage[utf8]{inputenc}
\usepackage[T1]{fontenc}
\usepackage[left= 2.2cm, right=2.2cm, top=2.5cm, bottom=2.5cm]{geometry}
\usepackage{amsmath,amssymb,amsthm}
\usepackage[]{authblk}
\usepackage{cite}
\usepackage{floatrow}
\usepackage{verbatim}
\usepackage{mathabx}
\usepackage{bbm}
\usepackage{float}
\usepackage{caption}
\usepackage{subcaption}
\usepackage{enumitem}
\usepackage{pgf,tikz}
\usetikzlibrary{arrows}
\usetikzlibrary[patterns]
\usepackage{hyperref}

\newtheorem{thm}{Theorem}
\newtheorem{cor}[thm]{Corollary}
\newtheorem{lem}[thm]{Lemma}
\newtheorem{prop}[thm]{Proposition}
\newtheorem{conj}[thm]{Conjecture}

\theoremstyle{definition}
\newtheorem{defn}[thm]{Definition}
\newtheorem{rem}[thm]{Remark}

\newtheorem{obs}[thm]{Observation}

\numberwithin{thm}{section}

\newcommand{\trel}{\ensuremath{T_{\mathrm{rel}}}}
\newcommand{\cA}{\ensuremath{\mathcal A}}
\newcommand{\cB}{\ensuremath{\mathcal B}}
\newcommand{\cC}{\ensuremath{\mathcal C}}
\newcommand{\cD}{\ensuremath{\mathcal D}}

\newcommand{\cK}{\ensuremath{\mathcal K}}
\newcommand{\cL}{\ensuremath{\mathcal L}}

\newcommand{\cS}{\ensuremath{\mathcal S}}

\newcommand{\cU}{\ensuremath{\mathcal U}}

\newcommand{\bbE}{{\ensuremath{\mathbb E}} }

\newcommand{\bbH}{{\ensuremath{\mathbb H}} }

\newcommand{\bbN}{{\ensuremath{\mathbb N}} }

\newcommand{\bbP}{{\ensuremath{\mathbb P}} }
\newcommand{\bbQ}{{\ensuremath{\mathbb Q}} }
\newcommand{\bbR}{{\ensuremath{\mathbb R}} }

\newcommand{\bbZ}{{\ensuremath{\mathbb Z}} }

\newcommand\Item[1][]{%
  \ifx\relax#1\relax  \item \else \item[#1] \fi
  \abovedisplayskip=0pt\abovedisplayshortskip=0pt~\vspace*{-\baselineskip}}

\renewcommand{\>}{\rangle}

\newcommand{\tbp}{{\ensuremath{\tau_{\mathrm{BP}}}} }
\newcommand{\diam}{{\ensuremath{\mathrm{diam}}} }
\newcommand{\qe}{{\ensuremath{q_{\mathrm{eff}}}} }
\newcommand{\1}{{\ensuremath{\mathbbm{1}}} }
\renewcommand{\o}{{\ensuremath{\omega}}}

\renewcommand{\leq}{\leqslant}
\renewcommand{\geq}{\geqslant}
\renewcommand{\le}{\leqslant}
\renewcommand{\ge}{\geqslant}
\renewcommand{\to}{\rightarrow}

\begin{document}
\title{Universality for critical KCM:\\infinite number of stable directions}
\author[,1,2]{Ivailo Hartarsky\thanks{\textsf{hartarsky@ceremade.dauphine.fr}}}
\author[,3]{Laure Mar\^ech\'e\thanks{\textsf{laure.mareche@epfl.ch}}}
\author[,2]{Cristina Toninelli\thanks{\textsf{toninelli@ceremade.dauphine.fr}}}
\affil[1]{DMA UMR 8553, \'Ecole Normale Sup\'erieure\protect\\CNRS, PSL Research University\protect\\45 rue d'Ulm, 75005 Paris, France}
\affil[2]{CEREMADE UMR 7534, Universit\'e Paris-Dauphine\protect\\CNRS, PSL Research University\protect\\Place du Mar\'echal de Lattre de Tassigny, 75775 Paris Cedex 16, France}
\affil[3]{LPSM UMR 8001, Universit\'e Paris Diderot\protect\\CNRS, Sorbonne Paris Cit\'e\protect\\
75013 Paris, France}
\date{\vspace{-0.25cm}\today}
\maketitle
\vspace{-0.75cm}
\begin{abstract}
Kinetically constrained models (KCM)  are reversible interacting particle systems on $\bbZ^d$ with continuous-time constrained Glauber dynamics. They are a natural non-monotone stochastic version of the family of cellular automata with random initial state known as $\cU$-bootstrap percolation. KCM have an interest in their own right, owing to their use for modelling the liquid-glass transition in condensed matter physics.

In two dimensions there are three classes of models with qualitatively different scaling of the infection time of the origin as the density of infected sites vanishes. Here we study in full generality the class termed `critical'. Together with the companion paper by Martinelli and two of the authors~\cite{Hartarsky19II} we establish the universality classes of critical KCM and determine within each class the critical exponent of the infection time as well as of the spectral gap. In this work we prove that for critical models with an infinite number of stable directions this exponent is twice the one of their bootstrap percolation counterpart. This is due to the occurrence of `energy barriers', which determine the dominant behaviour for these KCM but which do not matter for the monotone bootstrap dynamics. Our result confirms the conjecture of Martinelli, Morris and the last author~\cite{Martinelli19b}, who proved a matching upper bound.
\end{abstract}


\section{Introduction}
\label{sec:intro}

Kinetically constrained models (KCM) are interacting particle systems on the integer lattice $\bbZ^d$, which were introduced in the physics literature in the 1980s by Fredrickson and Andersen~\cite{Fredrickson84} in order to model the liquid-glass transition (see e.g.  \cite{Ritort03,Garrahan11} for reviews), a major and still largely open problem in condensed matter physics \cite{Berthier11}. A generic KCM is a continuous-time Markov process of Glauber type  characterised by  a finite collection $\cU$ of finite nonempty subsets of $\bbZ^d\setminus \{ 0 \}$, its \emph{update family}. A configuration $\o$ is defined by assigning to each site $x\in\bbZ^d$ an \emph{occupation variable} $\omega_x\in\{0,1\},$ corresponding to an \emph{empty} or \emph{occupied} site respectively. Each site $x\in\bbZ^d$ waits an independent, mean one, exponential time and then, iff  there exists $U\in \cU$ such that $\omega_y=0$ for all $y\in U+x$, site $x$ is updated to empty with probability $q$ and to occupied with probability $1-q$. Since each $U\in\mathcal U$ is contained in $\bbZ^d\setminus \{ 0 \}$, the constraint to allow the update does not depend on the state of the to-be-updated site. As a consequence, the dynamics satisfies detailed balance w.r.t. the product Bernoulli($1-q$) measure, $\mu$, which is therefore a reversible invariant measure. Hence the process started at $\mu$ is stationary.

Both from  a physical and from a mathematical point of view, a central issue for KCM is to determine the speed of divergence of the characteristic time scales when $q\to 0$. Two key quantities are: (i) the \emph{relaxation time} $\trel$, i.e. the inverse of the spectral gap of the Markov generator (see Definition~\ref{def:PC}) and (ii) the {\sl mean infection time} $\bbE(\tau_0)$, i.e. the mean over the stationary process started at $\mu$ of the first time at which the origin becomes empty. Several works have been devoted  to the study of these time scales for some specific choices of the constraints~\cite{Aldous02,Cancrini08,Mareche20Duarte,Chleboun14,Chleboun16,Martinelli19}  (see also \cite{Garrahan11} section 1.4.1 for a non exhaustive list of references in the physics literature). These results show that KCM exhibit a very large variety of possible scalings depending on the update family $\cU$. A question that naturally emerges, and that has been first addressed in \cite{Martinelli19b}, is whether it is possible to group all possible update families into distinct \emph{universality classes} so that all models of the same class display the same divergence of the time scales.

Before presenting the results and the conjectures of \cite{Martinelli19b}, we should describe  the key connection of KCM  with a class of discrete monotone cellular automata known as $\cU$-bootstrap percolation (or simply bootstrap percolation)~\cite{Bollobas15}.  For $\cU$-bootstrap percolation on $\bbZ^d$, given an update family $\mathcal{U}$ and a set $A_t$ of sites \emph{infected} at time $t$, the infected sites in $A_t$ remain infected at time $t+1$, and every site $x$ becomes infected at time $t + 1$ if the translate by $x$ of one of the sets in $\cU$ is contained in $A_t$. The set of initial infections $A$ is chosen at random with respect to the product Bernoulli measure with parameter $q \in [0,1]$, which identifies with $\mu$: for every $x\in\bbZ^d$ we have $\mu(x\in A)=q$. One then defines the  \emph{critical probability} $q_c\big( \bbZ^d, \cU \big)$ to be the infimum of the $q$ such that with probability one the whole lattice is eventually infected, namely  $\bigcup_{t\geq 0} A_t=\bbZ^d$. A key time scale for this dynamics is the first time at which the origin is infected, $\tbp$. In order to study this infection time for models on  $\bbZ^2$, the update families were classified by Bollob\'as, Smith and Uzzell~\cite{Bollobas15} into three universality classes: \emph{supercritical}, \emph{critical} and \emph{subcritical}, according to a simple geometric criterion (see Definition~\ref{def:preuniv}). In~\cite{Bollobas15} they proved that $q_c\big( \bbZ^2, \cU \big) = 0$ if $\cU$ is supercritical or critical, and it was proved by Balister, Bollob\'as, Przykucki and Smith~\cite{Balister16} that $q_c\big( \bbZ^2, \cU \big) > 0$ if $\cU$ is subcritical.
For supercritical update families, \cite{Bollobas15}  proved that $\tbp=q^{-\Theta(1)}$ w.h.p. as $q\to 0$, while in the critical case $\tbp=\exp(q^{-\Theta(1)})$. The result for critical families was later improved by Bollob\'as, Duminil-Copin, Morris and Smith~\cite{Bollobas14}, who identified the critical exponent $\alpha=\alpha(\cU)$ such that 
$\tbp=\exp(q^{-\alpha+o(1)})$.

Back to KCM, if we fix an update family $\cU$ and an initial configuration $\omega$ and we identify  the empty sites with infected sites, a first basic observation is that the clusters of sites that will never be infected in the $\cU$-bootstrap percolation correspond to clusters of sites which are occupied and will never be emptied under the KCM dynamics. A natural issue is whether there is a direct connection between the infection mechanism of bootstrap percolation and the relaxation mechanism for KCM, and, more precisely, whether the scaling of $\trel$ and $\bbE(\tau_0)$ is connected to the typical value of $\tbp$ when the law of the initial infections is $\mu$.
It is not difficult to establish that $\mu(\tbp)$ provides a lower bound for $\bbE(\tau_0)$ and $\trel$ (see \cite[Lemma~4.3]{Martinelli19} and \eqref{eq:trel:tau}), but in general, as we will explain, this lower bound does not provide the correct behaviour.

In \cite{Martinelli19b}, Martinelli, Morris and the last author proposed that the supercritical class should be refined into \emph{unrooted} supercritical and \emph{rooted} supercritical models in order to capture the richer behavior of KCM. For unrooted models the scaling is of the same type as for bootstrap percolation, $\trel\sim \bbE(\tau_0)=q^{-\Theta(1)}$ as $q \rightarrow 0$~\cite[Theorem~1(a)]{Martinelli19b}\footnote{For the lower bound of $\trel$ one does not need to use the boostrap percolation results, as $\trel\ge q^{-\min_{U\in\cU}|U|}/|\cU|$ by plugging the test function $\1_{\{\omega_0=0\}}$ in Definition~\ref{def:PC}.}, while for rooted models the divergence is much faster, $\bbE(\tau_0)\sim \trel =e^{\Theta((\log q)^2)}$ (see \cite[Theorem~1(b)]{Martinelli19b} for the upper bound and \cite[Theorem~4.2]{Mareche20Duarte} for the lower bound).

Concerning the critical class, the lower bound with $\mu(\tbp)$ mentioned above and the results of \cite{Bollobas15} on bootstrap percolation imply that $\trel$ and $\bbE(\tau_0)$ diverge at least as $\exp(q^{-\Theta(1)})$. In \cite[Theorem~2]{Martinelli19b} an upper bound of the same form was established and a conjecture~\cite[Conjecture~3]{Martinelli19b} was put forward on the value of the critical exponent $\nu$ such that both $\bbE(\tau_0)$ and $\trel$ scale as $\exp(|\log q|^{O(1)}/q^{\nu})$, with $\nu$ in general different from the exponent of the corresponding bootstrap percolation process. Furthermore, a toolbox was developed for the study of the upper bounds, leading to upper bounds matching this conjecture for all models. The main issue left open in \cite{Martinelli19b} was to develop tools to establish sharp lower bounds. A first step in this direction was done by Martinelli and the last two authors~\cite{Mareche20Duarte} by analyzing a specific critical model known as the Duarte model for which the  update family contains all the $2$-elements subsets of the North, South and West neighbours of the origin. Theorem~5.1 of \cite{Mareche20Duarte} establishes a sharp lower bound on the infection and relaxation times for the Duarte KCM that, together with the upper bound in \cite[Theorem~2(a)]{Martinelli19b}, proves $\bbE^{\mbox{\tiny{Duarte}}}(\tau_0)=\exp{(\Theta((\log q)^4/q^2))}$ as $q\to 0$, and the same result holds for $\trel$. The divergence is again much faster than for the corresponding bootstrap percolation model, for which it holds $\tbp=e^{\Theta((\log q)^2/q)}$ w.h.p as $q\to 0$~\cite{Mountford95} (see also \cite{Bollobas17}, from which the sharp value of the constant follows), namely the critical exponent for the Duarte KCM is twice the critical exponent for the Duarte bootstrap percolation. 

Both for Duarte and for supercritical rooted models, the sharper divergence of time scales for KCM is due to the fact that the infection time of KCM is not well approximated by the infection mechanism of the monotone bootstrap percolation process, but is instead the result of a much more complex infection/healing mechanism. Indeed, visiting regions of the configuration space with an anomalous amount of empty sites  is heavily penalised and requires a very long time to actually take place. 
The basic underlying idea is that the dominant relaxation mechanism is an \emph{East-like dynamics} for large \emph{droplets} of empty sites. Here East-like means that the presence of an empty droplet allows to empty (or fill) another adjacent droplet but only in a certain direction (or more precisely in a limited cone of directions). This is reminiscent of the relaxation mechanism for the East model, a prototype one-dimensional KCM for which $x$ can be updated iff $x-1$ is empty, thus  a single empty site allows to create/destroy an empty site  only on its right (see \cite{Faggionato13} for a review on the East model).
For supercritical rooted models, the empty droplets that play the role of the single empty sites for East have a finite (model dependent) size, hence an equilibrium density $q_{\mbox{\tiny{eff}}}= q^{\Theta(1)}$. For the Duarte model, droplets have a size that diverges  as $\ell=| \log q|/q$ and thus an equilibrium density $q_{\mbox{\tiny{eff}}}=q^{\ell}=e^{-(\log q)^2/q}$. Then a (very) rough understanding of the results of \cite{Mareche20Duarte,Martinelli19b} is obtained by replacing $q$ with $q_{\mbox{\tiny{eff}}}$ in the time scale for the East model $T_{\mathrm{rel}}^{\tiny\mbox{East}}=e^{\Theta((\log q)^2)}$\cite{Aldous02}.
The main technical difficulty to translate this intuition into a lower bound is that the droplets cannot be identified with a rigid structure. In \cite{Mareche20Duarte} this difficulty for the Duarte model was overcome by an algorithmic construction that allows to sequentially scan the system in search of sets of empty sites that \emph{could} (without violating the constraint) empty a certain rigid structure. These are the droplets that play the role of the empty sites for the East dynamics.

In \cite{Martinelli19b} all critical models which have an \emph{infinite number of stable directions} (see Section \ref{BP}), of which the Duarte model is but one example, were conjectured to have a critical exponent $\nu=2\alpha$, with $\alpha=\alpha(\cU)$ the critical exponent of the corresponding bootstrap percolation dynamics (defined in Definition~\ref{def:stable:alpha}). The heuristics is the same as for the Duarte model, the only difference being that droplets would have in general size $\ell=|\log q|^{O(1)}/q^{\alpha}$. However, the technique developed in \cite{Mareche20Duarte} for the Duarte model relies heavily on the specific form of the Duarte constraint and in particular on its oriented nature\footnote{Note that, since the Duarte update rules contain only the North, South and West neighbours of the origin, the constraint at a site $x$ does not depend on the sites with abscissa larger than the abscissa of $x$.}, and it cannot be extended readily to this larger class.

In this work, together with the companion paper by Martinelli and two of the authors \cite{Hartarsky19II}, we establish in full generality the universality classes for critical KCM, determining the critical exponent for each class.

Here we treat all choices of $\cU$ for which there is an \emph{infinite number of stable directions} and prove (Theorem \ref{th:main}) a lower bound for $\trel$ and $\bbE(\tau_0)$ that, together with the matching upper bound of \cite[Theorem~2]{Martinelli19b}, yields  \[\bbE(\tau_0)=e^{|\log q|^{O(1)}/q^{2\alpha}}\]
for $q\to 0$ and the same result for $\trel$. Our technique is somewhat inspired by the algorithmic construction of \cite{Mareche20Duarte}, however, the nature of the droplets which move in an East-like way is here much more subtle, and in order to identify them we construct an algorithm which can be seen as a significant improvement on the $\alpha$-covering and $u$-iceberg algorithms developed in the context of bootstrap percolation~\cite{Bollobas14}.

In the companion paper \cite{Hartarsky19II} we prove for the complementary class of models, namely all critical models with a \emph{finite number of stable directions}, an upper bound that (together with the lower bound from bootstrap percolation) yields instead
\[\bbE(\tau_0)=e^{|\log q|^{O(1)}/q^{\alpha}}\]
for $q\to 0$ and the same  result for $\trel$.

A comparison of our results with Conjecture~3 of \cite{Martinelli19b} is due. The class that we consider here is, in the notation of \cite{Martinelli19b}, the class of models with \emph{bilateral difficulty} $\beta=\infty$, hence belong to the \emph{$\alpha$-rooted} class defined therein. Therefore, our Theorem~\ref{th:main} proves Conjecture~3(a) in this case. We underline that it is not a limitation of our lower bound strategy that prevents us from proving Conjecture~3(a) for the other $\alpha$-rooted models, namely those with $2\alpha\le\beta<\infty$. Indeed, as it is proven in the companion paper~\cite{Hartarsky19II}, in this case the conjecture of~\cite{Martinelli19b} is not correct, since it did not take into account a subtle relaxation mechanism which allows to recover the same critical exponent as for the bootstrap percolation dynamics.

The plan of the paper is as follows. In Section~\ref{sec:background} we develop the background for both KCM and bootstrap percolation needed to state our result, Theorem~\ref{th:main}. In  Section~\ref{sec:sketch} we give a sketch of our reasoning and highlight the important points. In Section~\ref{sec:notation} we gather some preliminaries and notation. Section~\ref{sec:droplets} is the core of the paper --- there we define the central notions and establish their key properties, culminating in the Closure Proposition~\ref{prop:closure}. In Section~\ref{sec:East} we establish a connection between the KCM dynamics and an East dynamics and use this to wrap up the proof of Theorem~\ref{th:main}. Finally, in Section~\ref{sec:open} we discuss some open problems.

\section{Models and background}
\label{sec:background}
\subsection{Bootstrap percolation}\label{BP}
Before turning to our models of interest, KCM, let us recall recent universality results for the intimately connected bootstrap percolation models in two dimensions.
$\cU$-bootstrap percolation (or simply bootstrap percolation) is a very general class of monotone transitive local cellular automata on $\bbZ^{2}$ first studied in full generality by Bollob\'as, Smith and Uzzell~\cite{Bollobas15}. Let $\cU$, called \emph{update family}, be a finite family of finite nonempty subsets, called \emph{update rules}, of $\bbZ^2\setminus\{0\}$. Let $A$, called the set of \emph{initial infections}, be an arbitrary subset of $\bbZ^2$. Then the $\cU$-bootstrap percolation dynamics is the discrete time deterministic growth of infection defined by $A_0=A$  and, for each $t \in \bbN$,
\[A_{t+1}=A_t\cup\{x\in\bbZ^2 \colon \exists\, U\in\cU,U+x\subset A_t\}.\]
In other words, at any step each site becomes infected if a rule translated at it is already fully infected, and infections never heal. We define the \emph{closure} of the set $A$ by $[A]=\bigcup_{t\ge0}A_t$ and we say that $A$ is \emph{stable} when $[A]=A$. The set of initial infections $A$ is chosen at random with respect to the product Bernoulli measure $\mu$ with parameter $q \in [0,1]$: for every $x\in\bbZ^2$ we have $\mu(x\in A)=q$. 

Arguably, the most natural quantity to consider for these models is the typical (e.g. mean) value of $\tbp$, the infection time of the origin.

 The combined results of Bollobás, Smith and Uzzell~\cite{Bollobas15} and Balister, Bollobás, Przykucki and Smith~\cite{Balister16} yield a pre-universality partition of all update families into three classes with qualitatively different scalings of the median of the infection time as $q\to0$. In order to define this partition we will need a few definitions. 

For any unitary vector $u\in S^1=\{z\in\bbR^2 \colon \|z\|=1\}$ ($\|\cdot\|$ denotes the Euclidean norm in $\bbR^2$) and any vector $x\in\bbR^2$ we denote $\bbH_{u}(x)=\{y\in\bbR^2 \colon \<u,y-x\><0\}$ --- the open half-plane directed by $u$ passing through $x$. We also set $\bbH_u=\bbH_u(0)$. We say that a direction $u\in S^1$ is \emph{unstable} (for an update family $\cU$) if there exists $U\in\cU$ such that $U\subset \bbH_u$ and \emph{stable} otherwise. The partition is then as follows.
\begin{defn}[Definition~1.3 of~\cite{Bollobas15}]
\label{def:preuniv}
An update family $\cU$ is
\begin{itemize}
    \item \emph{supercritical} if there exists an open semi-circle of unstable directions,
    \item \emph{critical} if it is not supercritical, but there exists an open semi-circle with a finite number of stable directions,
    \item \emph{subcritical} otherwise.
\end{itemize}
\end{defn}

The main result of~\cite{Bollobas15} then states that in the supercritical case $\tbp=q^{-\Theta(1)}$ with high probability as $q\to 0$, while in the critical one $\tbp=\exp(q^{-\Theta(1)})$. The final justification of the partition in Definition~\ref{def:preuniv} was given by Balister, Bollobás, Przykucki and Smith~\cite{Balister16} who proved that the origin is never infected with positive probability for subcritical models for $q>0$  sufficiently small, i.e. $q_c\big( \bbZ^2, \cU \big) > 0$ if $\mathcal{U}$ is subcritical. From the bootstrap percolation perspective supercritical models are rather simple, while subcritical ones remain very poorly understood (see~\cite{Hartarsky18subcritical}). Nevertheless, most of the non-trivial models considered before the introduction of $\cU$-bootstrap percolation, including the $2$-neighbour model  (see~\cite{Aizenman88,Holroyd03} for further results), fall into the critical class, which is also the focus of our work.

Significantly improving the result of \cite{Bollobas15}, Bollobás, Duminil-Copin, Morris and Smith~\cite{Bollobas14} found the correct exponent determining the scaling of $\tbp$ for critical families.
Moreover, they were able to find $\log \tbp$ up to a constant factor. To state their results we need the following crucial notion.

\begin{defn}[Definition~1.2 of~\cite{Bollobas14}]
\label{def:stable:alpha}
Let $\cU$ be an update family and $u\in S^1$ be a direction. Then the \emph{difficulty} of $u$, $\alpha(u)$, is defined as follows.
\begin{itemize}
    \item If $u$ is unstable, then $\alpha(u)=0$.
    \item If $u$ is an isolated stable direction (isolated in the topological sense), then
    \begin{equation}
    \label{eq:def:alpha:u}
    \alpha(u)=\min\{n\in\bbN \colon \exists K\subset\bbZ^2,|K|=n,|[\bbZ^2\cap(\bbH_u\cup K)]\setminus\bbH_u|=\infty\},
    \end{equation}
    i.e. the minimal number of infections allowing $\bbH_u$ to grow infinitely.
    \item Otherwise, $\alpha(u)=\infty$.
\end{itemize}
We define the \emph{difficulty} of $\cU$ by
\begin{equation}
\label{eq:def:alpha}
\alpha(\cU)=\inf_{C\in\cC}\sup_{u\in C}\alpha(u),
\end{equation}
where $\cC=\{\bbH_u\cap S^1 \colon u\in S^1\}$ is the set of open semi-circles of $S^1$. 
\end{defn}

It is not hard to see (Theorem~1.10 of~\cite{Bollobas15}, Lemma~2.6 of~\cite{Bollobas14}) that the set of stable directions is a finite union of closed intervals of $S^1$ and that (Lemmas~2.7 and 2.10 of~\cite{Bollobas14}) \eqref{eq:def:alpha:u} also holds for unstable and \emph{strongly stable} directions, that is directions in the interior of the set of stable directions (but not for \emph{semi-isolated stable} directions i.e. endpoints of non-trivial stable intervals). Furthermore (see \cite[Lemma~2.7]{Bollobas14}, \cite[Lemma~5.2]{Bollobas15}), $1\le \alpha(u)<\infty$ if and only if $u$ is an isolated stable direction, so that $\cU$ is critical if and only if $1\le\alpha(\cU)<\infty$. 
As a final remark we recall that, contrary to determining whether an update family is critical, finding $\alpha(\cU)$ is a NP-hard question~\cite{Hartarsky18NP}.

We are now ready to describe the universality results. A weaker form of the result of~\cite{Bollobas14} is that $\tbp=\exp(q^{-\alpha(\cU)+o(1)})$ with high probability as $q \rightarrow 0$. 
For the full result however, we need one last definition.
\begin{defn}
\label{def:balanced}
A critical update family $\cU$ is \emph{balanced} if there exists a closed semi-circle $C$ such that $\max_{u\in C}\alpha(u)=\alpha(\cU)$ and \emph{unbalanced} otherwise.
\end{defn}
Then~\cite{Bollobas14} provides that for balanced models $\tbp=\exp(\Theta(1)/q^{\alpha(\cU)})$ with high probability as $q \to 0$, while for unbalanced ones $\tbp=\exp(\Theta((\log q)^2)/q^{\alpha(\cU)})$. These are the best general estimates currently known. 
We refer to~\cite{Morris17,Morris17b} for recent surveys on these results as well as on sharper results for some  specific models.

\subsection{Kinetically constrained models}
\label{sec:KCM}
Returning to KCM, let us first define the general class of KCM introduced by Cancrini, Martinelli, Roberto and the last author~\cite{Cancrini08} directly on $\bbZ^2$. Fix a parameter $q\in[0,1]$ and an update family $\cU$ as in the previous section. The corresponding KCM is a continuous-time Markov process on $\Omega=\{0,1\}^{\bbZ^2}$ which can be informally defined as follows. A configuration $\omega$ is defined by assigning to each site $x\in\bbZ^2$ an \emph{occupation variable} $\omega_x\in\{0,1\}$ corresponding to an \emph{empty} (or \emph{infected}) and \emph{occupied} (or \emph{healthy}) site respectively. Each site waits an independent exponentially distributed time with mean $1$ before attempting to update its occupation variable. At that time, if the configuration is completely empty on at least one update rule translated at $x$, i.e. if $\exists U\in\cU$ such that $\omega_y=0$ for all $y\in U+x$, then we perform a \emph{legal update} or \emph{legal spin flip} by setting $\omega_x$ to $0$ with probability ${q}$ and to $1$ with probability ${1-q}$. Otherwise the update is discarded. Since the constraint to allow the update never depends on the state of the to-be-updated site, the product measure $\mu$ is a reversible invariant measure and the process started at $\mu$ is stationary. More formally, the KCM is the Markov process on $\Omega$ with generator $\cL$ acting on local functions $f\colon \Omega \mapsto \bbR$  as
\begin{equation}
\label{eq:generator}
(\cL f)(\omega)= \sum_{x\in \bbZ^2}c_x(\omega)\left(\mu_x(f) - f \right)(\omega),
\end{equation}
for any $\omega \in \Omega$, where $\mu_x(f)$ denotes the average of $f$ when the occupation variable at $x$ has law $\mathrm{Ber}(1-q)$ and the other occupation variables are set to $\{\omega_y\}_{y\neq x}$, and $c_x$ is the indicator function of the event that there exists $U\in \cU$ such that $U+x$ is completely empty, i.e. $\omega_{U+x}\equiv 0$. We refer the reader to chapter~I of~\cite{Liggett05}, where the general theory of interacting particle systems is detailed, for a precise construction of the Markov process and the proof that $\cL$ is the generator of a reversible Markov process $\{\omega(t)\}_{t\geq 0}$ on $\Omega$ with reversible measure $\mu$.

The corresponding Dirichlet form is defined as
\begin{equation}\label{def:Dirichlet}
\cD(f)= \sum_{x\in \bbZ^2}\mu\big(c_x \mbox{Var}_x(f)\big),
\end{equation}
where $\mbox{Var}_x(f)$ denotes the variance of the local function $f$ with respect to the variable
$\omega_x$ conditionally on $\{\omega_y\}_{y\neq x}$.
The expectation with respect to the stationary process
with initial distribution $\mu$ will be denoted by $\bbE=\bbE_{\mu}^{q,\cU}$. Finally, given a configuration $\omega\in\Omega$ and a site $x\in\bbZ^2$, we will denote by $\omega^x$
the configuration obtained from $\omega$ 
by \emph{flipping} site $x$, namely by setting $(\omega^x)_x=1-\omega_x$ and $(\omega^x)_y=\omega_y$ for all $y\neq x$. For future use we also need the following definition of legal paths, that are essentially sequences of configurations obtained by successive legal updates.
\begin{defn}[Legal path] 
\label{def:legal}Fix an update family $\cU$, then a \emph{legal path} $\gamma$ in $\Omega$ is a finite sequence $\gamma=\big(\omega_{(0)},\dots,\omega_{(k)}\big)$ such that, for each $i\in \{1,\dots,k\},$ the configurations $\omega_{(i-1)}$ and $\omega_{(i)}$ differ by a legal (with respect to the choice of $\cU$) spin flip at some vertex $v= v(\omega_{(i-1)},\omega_{(i)})$.
\end{defn}

As mentioned in Section~\ref{sec:intro}, our goal is to prove sharp bounds on the characteristic time scales of critical KCM. Let us start by defining precisely these time scales, namely the relaxation time $\trel$ (or inverse of the spectral gap) and the mean infection time $\bbE(\tau_0)$ (with respect to the stationary process).

\begin{defn}[Relaxation time $\trel$]
\label{def:PC} Given an update family $\cU$ and $q\in[0,1]$,
we say that $C>0$ is a \emph{Poincar\'e constant} for the corresponding KCM if, for all local functions $f$, we have
\begin{equation}
  \label{eq:gap}
\mbox{Var}_{\mu}(f)=\mu(f^2)-\mu(f)^2 \leq C \, \cD(f).
\end{equation}
If there exists a finite Poincar\'e constant, we define
\[
\trel=\trel(q,\cU)=\inf\left\{ C > 0 \,:\, C \text{ is a Poincar\'e constant} \right\}.
\]
Otherwise we say that the relaxation time is infinite.
\end{defn}
A finite relaxation time implies that the reversible measure $\mu$ is mixing for the semigroup $P_t=e^{t\cL}$ with exponentially decaying time auto-correlations (see e.g. \cite[Section~2.1]{Bakry06}).

\begin{defn}[Infection time $\tau_0$]
The random time $\tau_0$ at which the origin is first infected is given by
\[\tau_0 = \inf\big\{ t \ge 0 \,:\, \omega_0(t)=0 \big\},\]
where we adopt the usual notation letting  $\omega_0(t)$ be the value of the configuration $\omega(t)$ at the origin, namely $\omega_0(t) =(\omega(t))_0$.
\end{defn} 

\paragraph{The East model}
We close this section by defining a specific example of KCM on $\bbZ$, the East model of J\"ackle and Eisinger~\cite{Jackle91}, which
will be crucial to understand our results (KCM on $\bbZ$ are defined in the same way as KCM on $\bbZ^2$).
It 
is defined by an update family composed by a single rule containing only the site to the left of the origin ($-1$). In other words, site $x$ can be updated iff $x-1$ is empty. For this model both $\trel$ and $\bbE(\tau_0)$ scale as $\exp\left(\frac{(\log q)^2}{2\log 2}\right)$ as $q\to 0$, see  \cite{Aldous02,Cancrini08,Chleboun14}\footnote{Actually these references focus on the study of $\trel$. A matching upper bound for $\bbE(\tau_0)$ follows 
from \eqref{eq:trel:tau}. The lower bound for $\bbE(\tau_0)$ follows easily from the lower bound for $\bbP(\tau_0>t)$ with $t=\exp{(\log(q)^2/2\log 2)}$ obtained in the proof of Theorem~5.1 of \cite{Cancrini10}.}. One of the key ingredients behind this scaling is the following combinatorial result~\cite{Sollich99} (see \cite[Fact 1]{Chung01} for a more mathematical formulation).
\begin{prop}
\label{lem:comb_East}
Consider the East model on $\{1,\dots,M\}$ defined by fixing $\omega_{0}=0$ at all time. Then any legal path $\gamma$ connecting the fully occupied configuration (namely $\omega$ s.t. $\omega_x=1$ for all $x\in\{1,\dots,M\}$) to a configuration $\omega'$ such that $\omega'_M=0$ goes through a configuration with at least $\lceil\log_2 (M+1)\rceil$ empty sites.
\end{prop}
This logarithmic `energy barrier', to employ the physics jargon, and the fact that at equilibrium the typical distance to the first empty site is $M=\Theta(1/q)$ are responsible for the divergence of the time scales as roughly $1/q^{\lceil\log_2(M+1)\rceil}=e^{\Theta((\log q)^2)}$.

\subsection{Result}
In this paper we study critical KCM with an infinite number of stable directions or, equivalently, with a non-trivial interval of stable directions. Recall that $\bbE$ denotes the expectation with respect to the stationary KCM process.
\begin{thm}
\label{th:main}
Let $\cU$ be a critical update family with an infinite number of stable directions. Then there exists a sufficiently large constant $C>0$ such that
\[\bbE(\tau_0)\ge \exp\left(1/\left(Cq^{2\alpha(\cU)}\right)\right),\]
as $q\to0$ and the same asymptotics holds for $T_{\mathrm{rel}}$.
\end{thm}
This theorem combined with the upper bound of Martinelli, Morris and the last author~\cite[Theorem~2(a)]{Martinelli19b}, determines the critical exponent of these models to be $2\alpha$ in the sense of Corollary~\ref{cor:univ} below. We thus complete the proof of universality and Conjecture~3(a) of~\cite{Martinelli19b} for these models\footnote{The conjecture involuntarily asks for a positive power of $\log q$, which we do not expect to be systematically present (see Conjecture~\ref{conj:logs}).}.
\begin{cor}
\label{cor:univ}
Let $\cU$ be a critical update family with an infinite number of stable directions. Then
\[q^{2\alpha(\cU)}\log\bbE(\tau_0)=(-\log q)^{O(1)}\]
as $q\to0$ and the same holds for $\trel$.
\end{cor}

Universality for the remaining critical models is proved in a companion paper by Martinelli and the first and third authors~\cite{Hartarsky19II} and, in particular, Conjecture~3(a) of~\cite{Martinelli19b} is disproved for models other than those covered by Theorem~\ref{th:main}.
It is important to note that Theorem~\ref{th:main} significantly improves the best known results for all models  with the exception of the recent result of Martinelli and the last two authors~\cite{Mareche20Duarte} for the Duarte model. Indeed, the previous bound had exponent $\alpha$, and was proved via the general (but in this case far from optimal) lower bound with the mean infection time for the corresponding bootstrap percolation model~\cite[Lemma~4.3]{Martinelli19}.

\section{Sketch of the proof}
\label{sec:sketch}
In this section we outline roughly the strategy to derive our main result, Theorem~\ref{th:main}. The hypothesis of infinite number of stable directions provides us with an interval of stable directions. We can then construct stable `droplets' of shape as in Figure~\ref{fig:DYD} (see Definitions~\ref{def:DYD} and \ref{def:CDYD}), where we recall from Section~\ref{BP} that a set is stable if it coincides with its closure. Thus, if all infections are initially inside a droplet, this will be true at any time under the KCM dynamics. The relevance and advantage of such shapes come from the fact that only infections situated to the left of a droplet can induce growth left. This is manifestly not feasible without the hypothesis of having an interval of stable directions.
It is worth noting that these shapes, which may seem strange at first sight, are actually very natural and intrinsically present in the dynamics. Indeed, such is the shape of the stable sets for a representative model of this class -- the modified 2-neighbour model with one (any) rule removed, that is the three-rule update family with rules $\{(-1,0),(0,1)\}$,$\{(-1,0),(0,-1)\}$,$\{(0,-1),(1,0)\}$ (it can also be seen as the modified Duarte model with an additional rule). The stable sets in this case are actually  Young diagrams.

We construct a collection of such droplets covering the initial configuration of infections, so that it gives an upper bound on the closure. To do this, we devise an improvement of the $\alpha$-covering algorithm of Bollobás, Duminil-Copin, Morris and Smith~\cite{Bollobas14}. It is important for us not to overestimate the closure as brutally. Indeed, a key step and the main difficulty of our work is the Closure Proposition~\ref{prop:closure}, which roughly states that the collections of droplets associated to the closure of the initial infections is equal to the collection for the initial infections. This is highly non-trivial, as in order not to overshoot in defining the droplets, one is forced to ignore small patches of infections (larger than the ones in~\cite{Bollobas14}), which can possibly grow significantly when we take the closure for the bootstrap percolation process and especially so if they are close to a large infected droplet. In order to remedy this problem, we introduce a relatively intrinsic notion of `crumb' (see Definition~\ref{def:crumb}) such that its closure remains one and does not differ too much from it. A further advantage of our algorithm for creating the droplets over the one of~\cite{Bollobas14} is that it is somewhat canonical, with a well-defined unique output, which has particularly nice `algebraic' description and properties (see Remark~\ref{rem:algo}). Another notable difficulty we face is systematically working in roughly a half-plane (see Remark~\ref{rem:generalisations:algo} for generalisations) with a fully infected boundary condition, but we manage to extend our reasoning to this setting very coherently.

Finally, having established the Closure Proposition~\ref{prop:closure} alongside standard and straightforward results like an Aizenmann-Lebowitz Lemma~\ref{lem:AL} and an exponential decay of the probability of occurrence of large droplets (Lemma~\ref{lem:exp:decay:diam}), we finish the proof via the following approach,
inspired by the one 
developed by Martinelli and the last two authors~\cite{Mareche20Duarte} for the Duarte model.
The key step here (see Section \ref{sec:East}) is mapping the KCM legal paths to those of an East dynamics via a suitable renormalisation. 
 Roughly speaking, we say that a renormalised site is empty if it contains a large droplet of infections. However, for the renormalised configuration to be mostly invariant under the original KCM dynamics, we rather look for the droplets in the closure of the original set of infections instead. This is where the Closure Proposition~\ref{prop:closure}  is used to compensate the fact that the closure of equilibrium is not equilibrium. 
In turn, this mapping  together with the combinatorial result for the East model recalled in Section~\ref{sec:KCM} (Proposition \ref{lem:comb_East}), yield a bottleneck for our dynamics corresponding to the creation of $\log(1/\qe)$ droplets, where $1/\qe$ is the equilibrium distance between two empty sites in the renormalized lattice, and $\qe\sim e^{-1/q^{\alpha}}$. This provides for the time scales the desired lower bound $q_{\mathrm{eff}}^{\log(\qe)}\sim e^{1/q^{2\alpha}}$ of Theorem~\ref{th:main}. The last part of the proof follows very closely the ideas put forward 
in~\cite{Mareche20Duarte} for the Duarte model. However, in~\cite{Mareche20Duarte}, there was no need to develop a subtle droplet algorithm since, owing to the oriented character of the Duarte constraint, droplets could simply be identified with some large infected vertical segments. It is also worth noting that, thanks to the less rigid notion of droplets that we develop in the general setting, some of the difficulties faced in~\cite{Mareche20Duarte} for Duarte are no longer present here.

\section{Preliminaries and notation}
\label{sec:notation}
Let us fix a critical update family $\cU$ with an infinite number of stable directions for the rest of the paper. We will omit $\cU$ from all notation, such as $\alpha(\cU)$.

\paragraph{Directions}
The next lemma establishes that one can make a suitable choice of $4$ stable directions, which we will use for all our droplets. At this point the statement should look very odd and technical, but it simply reflects the fact that we have a lot of freedom for the choice and we make one which will simplify a few of the more technical points in later stages. Nevertheless, this is to a large extent not needed besides for concision and clarity.

\begin{figure}[t]
\floatbox[{\capbeside\thisfloatsetup{capbesideposition={right,top}}}]{figure}[\FBwidth]{
\begin{tikzpicture}[line cap=round,line join=round,x=2.0cm,y=2.0cm]
\clip(-1.05,-1.05) rectangle (1.05,1.05);
\draw(0,0) circle (2cm);
\draw [->] (0,0) -- (-0.99,0.17);
\draw [->] (0,0) -- (-0.99,-0.17);
\draw [->] (0,0) -- (0.78,-0.63);
\draw [->] (0,0) -- (-0.16,0.99);
\draw [->,dash pattern=on 1pt off 1pt] (0,0) -- (0.99,-0.17);
\draw [->,dash pattern=on 1pt off 1pt] (0,0) -- (0.99,0.17);
\draw [shift={(0,0)},line width=2pt,color=red]  plot[domain=2.78:3.62,variable=\t]({1*1*cos(\t r)+0*1*sin(\t r)},{0*1*cos(\t r)+1*1*sin(\t r)});
\draw [shift={(0,0)},line width=2pt,color=red]  plot[domain=1.73:2.34,variable=\t]({1*1*cos(\t r)+0*1*sin(\t r)},{0*1*cos(\t r)+1*1*sin(\t r)});
\begin{scriptsize}
\draw[color=black] (-0.5,0.15) node {$u_1$};
\draw[color=black] (-0.5,-0.15) node {$u_2$};
\draw[color=black] (0.4,-0.45) node {$v_1'$};
\draw[color=black] (0,0.5) node {$v_2'$};
\draw[color=black] (0.6,-0.05) node {$u_1+\pi$};
\draw[color=black] (0.55,0.2) node {$u_2-\pi$};
\fill [color=red] (-0.16,0.99) circle (1.5pt);
\fill [color=red] (0.78,-0.63) circle (1.5pt);
\fill [color=red] (-0.32,-0.95) circle (1.5pt);
\fill [color=red] (0.88,0.48) circle (1.5pt);
\fill [color=red] (-0.69,0.72) circle (1.5pt);
\fill [color=red] (-0.94,0.35) circle (1.5pt);
\fill [color=red] (-0.89,-0.46) circle (1.5pt);
\draw[color=black] (0.95,0.48) node {$1$};
\draw[color=black] (0.85,-0.63) node {$2$};
\draw[color=black] (-0.4,-1.0) node {$3$};
\end{scriptsize}
\end{tikzpicture}}
{\caption{Illustration of Lemma~\ref{lem:directions} and its proof. Thickened arcs represent intervals of strongly stable directions. Solid dots represent isolated and semi-isolated stable directions. The difficulties of the isolated stable directions are indicated next to them and yield that the difficulty of the model is $\alpha=2$. The directions chosen in Lemma~\ref{lem:directions} are the solid vectors $u_1$, $u_2$, $v_1=v_1'$ and a direction $v_2$ in the strongly stable interval ending at $v_2'$ sufficiently close to $v_2'$. Note that the definition of $v_2'$ (and $v_1'$) disregards stable directions with difficulty smaller than $\alpha$ as present on the figure.
}
\label{fig:directions}}
\end{figure}
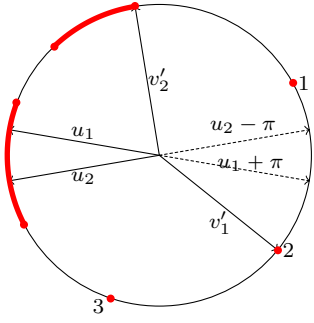
A direction $u\in S^1$ is called \emph{rational} if $\tan u\in\bbQ\cup\{\infty\}$.
\begin{lem}
\label{lem:directions}
There exists rational stable directions $\cS=\{u_1,u_2,v_1,v_2\}$ (see Figure~\ref{fig:directions}) with difficulty at least $\alpha$ such that
\begin{itemize}
\item The directions appear in couterclockwise order $u_1,u_2,v_1,v_2$.
\item No $u\in\cS$ is a semi-isolated stable direction.
\item $u_{3-i}$ belongs to the cone spanned by $v_i$ and $u_i$ for $i\in\{1,2\}$ i.e. the strictly smaller interval among $[v_i,u_i]$ and $[u_i,v_i]$ contains $u_{3-i}$. 
\item $0$ is contained in the interior of the convex envelope of $\cS$.
\item Either $u_2<v_1-\pi/2$ or $u_1>v_2+\pi/2$.
\item $(\bbH_{u_1}\cup\bbH_{u_2})\cap\bbZ^2$ is stable or, equivalently, $\nexists U\in\cU, U\subset\bbH_{u_1}\cup\bbH_{u_2}$.
\item the directions
\begin{equation*}
\begin{split}
u'=&(u_1+u_2)/2,\\
u'_1=&(3u_1+u_2)/4,\\
u'_2=&(u_1+3u_2)/4
\end{split}
\end{equation*}
are rational.
\end{itemize}
\end{lem}
\begin{proof}
Since $\cU$ has an infinite number of stable directions and they form a finite union of closed intervals with rational endpoints~\cite[Theorem~1.10]{Bollobas15}, there exists a non-empty open interval $I'''$ of stable directions. Further note that the set $J$ of directions $u$ such that there exists a rule $U\in\cU$ and $x\in U$ with $\<x,u\>=0$ is finite, so one can find a non-trivial closed subinterval $I''\subset I'''$ which does not intersect $J$. The directions $u_1$ and $u_2$ will be chosen in $I''$, which clearly implies that they are strongly stable and thus with infinite difficulty.  Moreover, if there exists $U\in\cU$ with $U\subset \bbH_{u_1}\cup\bbH_{u_2}$, by stability of $u_2$, we have $U\cap(\bbH_{u_1}\setminus\bbH_{u_2})\neq\varnothing$, which contradicts $I''\cap J=\varnothing$.

Since $\cU$ is critical it does not have two opposite strongly stable directions, so there is no strongly stable direction in $I''+\pi$. If there are any (isolated or semi-isolated) stable directions in $I''+\pi$, we can further choose a non-trivial open subinterval $I'\subset I''$, for which this is not the case (there is a finite number of isolated and semi-isolated stable directions). Let $\pi>\delta>0$ be such that the angle between any two consecutive directions of difficulty at least $\alpha$ is at most $\pi-\delta$ (it is well defined by~\eqref{eq:def:alpha}). We then choose a non-trivial closed subinterval $I'\supset I=[u_1,u_2]$ with $u_1$ rational and $u'_1=(3u_1+u_2)/4$ rational and with $0<u_2-u_1<\delta<\pi$. It easily follows from the sum and difference formulas for the tangent function that $u'$, $u'_2$ and $u_2$ are also rational.

Let 
\begin{align*}
v_1'=&\max\{v\in(u_2,u_1+\pi)\colon\alpha(v)\ge\alpha\},\\
v_2'=&\min\{v\in(u_2-\pi,u_1)\colon\alpha(v)\ge\alpha\}.
\end{align*}
These both exist, since $I+\pi$ does not contain stable directions, both $(u_2,u_2+\pi)$ and $(u_1-\pi,u_1)$ contain directions with difficulty at least $\alpha$ by~\eqref{eq:def:alpha} and the set of such directions is closed. If $v_1'$ is not semi-isolated, we set $v_1=v_1'$ and similarly for $v_2$. Otherwise, we choose a rational strongly stable direction sufficiently close to $v_1'$ as $v_1$ and similarly for $v_2$. We claim that this choice satisfies all the desired conditions. Indeed, all directions in $\cS$ are stable non-semi-isolated rational with difficulty at least $\alpha$ and the last but one condition was already verified.

One does have that $u_1$ is in the cone spanned by $v_2$ and $u_2$, which is implied by $v_2\in(u_2-\pi,u_1)$ and similarly for $u_2$, so the third condition is also verified. If $v_2'-v'_1\ge \pi$, then there is an open half circle contained in $(v_1',v_2')$ with no direction of difficulty at least $\alpha$, which contradicts~\eqref{eq:def:alpha}, so $v_2-v_1<\pi$ and the same holds for $u_1-v_2$, $u_2-u_1$ and $v_1-u_2$ by the definition of $v_1'$ and $v_2'$, the fact that $v_1$ and $v_2$ are sufficiently close to them and the fact that $I$ was chosen smaller than $\pi$. Thus $0$ is in the convex envelope of $\cS$.

Finally, if one has both $v_1-u_2\le \pi/2$ and $u_1-v_2\le\pi/2$, then one obtains $v_2'-v_1'>\pi-\delta$, since $I$ is smaller than $\delta$. However, $v_1'$ and $v_2'$ are consecutive directions of difficulty at least $\alpha$, which contradicts the definition of $\delta$.
\end{proof}
\paragraph{Notation} For the rest of the paper we fix directions $\cS=\{u_1,u_2,v_1,v_2\}$ as in Lemma~\ref{lem:directions} and assume without loss of generality that $u_2<v_1-\pi/2$.

Let us fix large constants
\[1\ll C_1\ll C_2'\ll C_2\ll C_3\ll C_4'\ll C_4\ll C_5,\]
each of which can depend on previous ones as well as on $\cU$ and $\cS$. We will also use asymptotic notation whose constants can depend on $\cU$ and $\cS$, but not on $C_1$ or the other constants above. All asymptotic notation is with respect to $q\to0$, so we assume throughout that $q>0$ is sufficiently small.

For any two sets $K,\partial\subset\bbR^2$ we define $[K]_\partial=[(K\cup\partial)\cap\bbZ^2]\setminus\partial$.

Finally, we make the convention that throughout the article all distances, balls and diameters are Euclidean unless otherwise stated. We say that a set $X\subset\bbR^2$ is \emph{within distance} $\delta$ of a set $Y\subset\bbR^2$ if $d(x,Y)\le\delta$ for all $x\in X$\, where $d$ is the Euclidean distance.

\section{Droplet algorithm}
\label{sec:droplets}

In this section we define our main tool -- the droplet algorithm. It can be seen as a significant improvement on the $\alpha$-covering and $u$-iceberg algorithms~\cite[Definitions~6.6 and~6.22]{Bollobas14}, many of whose techniques we adapt to our setting.

We will work in an infinite domain $\Lambda$ defined as follows (see Figure~\ref{fig:Lambda}).
Fix some vector $a_0\in\bbR^2$ and let
\begin{equation}
\label{eq:def:partial:Lambda}
\begin{split}
\partial=&\bbH_{u'}\cup\bbH_{u'_1}(a_0)\cup\bbH_{u'_2}(a_0),\\
\Lambda=&\bbR^2\setminus\partial,
\end{split}
\end{equation}
where the directions $u'$, $u'_1$ and $u'_2$ are those defined in Lemma~\ref{lem:directions}.
In other words, $\Lambda$ is a cone with sides perpendicular to $u_1'$ and $u_2'$ cut along a line perpendicular to $u'$.  The reader is invited to simply think that $\partial$ is a half-plane directed by $u'$, which will not change the reasoning. 

\begin{figure}[t]
\floatbox[{\capbeside\thisfloatsetup{capbesideposition={right,top}}}]{figure}[\FBwidth]{
\begin{tikzpicture}[line cap=round,line join=round,x=0.25cm,y=0.25cm]
\clip(-5,-10) rectangle (3,10);
\fill[fill=black,fill opacity=0.25] (-5,-10.5) -- (0,-3) -- (0,3) -- (-5,10.5) -- (25,25) -- (25,-25) -- cycle;
\draw [->] (0,1) -- (-2,1);
\draw [->] (-2,6) -- (-3.66,4.89);
\draw [->] (-2,-6) -- (-3.66,-4.89);
\draw [domain=-5.0:5.0] plot(\x,{(-3--1.5*\x)/-1});
\draw [domain=-5.0:5.0] plot(\x,{(--3-1.5*\x)/-1});
\draw (0,-20) -- (0,20);
\begin{scriptsize}
\fill [color=black] (2,0) circle (1pt);
\draw[color=black] (2,1) node {$a_0$};
\draw[color=black] (-1,8) node {$\partial$};
\draw[color=black] (-4,0.61) node {$\Lambda$};
\draw[color=black] (-1,0) node {$u'$};
\draw[color=black] (-2,-4.5) node {$u'_1$};
\draw[color=black] (-2,4.5) node {$u'_2$};
\end{scriptsize}
\end{tikzpicture}}
{\caption{The open domain $\partial$ defined in~\eqref{eq:def:partial:Lambda} is shaded, while its complement $\Lambda$ is not. The lines are the boundaries of the three half-planes defining $\partial$. Note that if $a_0\notin\bbH_{u'}$, then $\Lambda$ becomes simply a cone. }
\label{fig:Lambda}}
\end{figure}
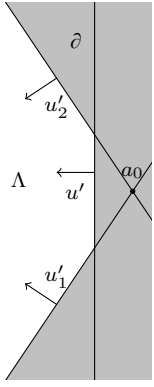

\subsection{Clusters and crumbs}
Let $\Gamma$ be the graph with vertex set $\bbZ^2$ but with $x\sim y$ if and only if $\|x-y\|\le C_2$. Let $\Gamma'$ be defined similarly with $C_2$ replaced by $C_2'$. Given a finite $K\subset \Lambda\cap\bbZ^2$, we say that $\kappa\subset K$ is a \emph{connected component} of $K$ in $\Gamma$ if the subgraph of $\Gamma$ induced by the vertex set $\kappa$ is connected and there do not exist vertices $x\in K\setminus\kappa$ and $y\in \kappa$ such that $x\sim y$ in $\Gamma$.

\paragraph{Crumbs}
For a given finite set $K\subset\Lambda\cap\bbZ^2$ of infections we would like to have a notion of a connected component being `big' or `small'. `Small' components will be dubbed `crumbs' and will play a negligible perturbative role in the bootstrap percolation process, by inducing only `very localised' growth and being `well isolated' from the rest of the infections. A sufficient condition for this, as identified in \cite{Bollobas14}, is that $|\kappa|<\alpha$. However, contrary to what was the case in \cite{Bollobas14}, we need the notion of `crumb' to be stable under the closure (with respect to the bootstrap percolation process), i.e. the closure of a `crumb' to still be a `crumb'. We thus identify as `crumb' any component, which is the closure of a set of size less than $\alpha$. Also taking into account the boundary, this leads us to the following notion.

\begin{defn}[Crumb]
\label{def:crumb}
Fix a finite set $K\subset \Lambda\cap\bbZ^2$ and let $\kappa$ be a connected component of $K$ in $\Gamma$. We say that $\kappa$ is a \emph{crumb} for $K$ if the following conditions hold.
\begin{itemize}
    \item For all $x\in\kappa$ we have $d(x,\partial)>C_2$.
    \item There exists a set $P_\kappa\subset \bbZ^2$ such that $[P_\kappa]\supset \kappa$ and $|P_\kappa|=\alpha-1$.
\end{itemize}
\end{defn}
\paragraph{First properties of crumbs}
It follows from the definition that a crumb $\kappa$ for $K$ is at distance more than $C_2$ from $\partial \cup (K\setminus\kappa)$. Moreover, the closure of a crumb is within bounded distance from the crumb, as we shall see in Corollary~\ref{cor:crumbs:growth} (see Figure \ref{fig:crumbs}). Also, crumbs have diameters much smaller than $C_3$, as we shall see in Corollary~\ref{cor:crumbs:growth}. The proofs of this corollary and Observation~\ref{obs:C1:well:defined}, which it follows from, are both independent of the rest of the argument and are only postponed for convenience. Nevertheless, we allow ourselves to use these (easy) results ahead of their proofs.

These properties justify and quantify the intuition that crumbs are `small', that they only grow `locally', and it is clear that (if we disregard the boundary) the closure of a crumb is a crumb.
\paragraph{Modified crumbs}
Unfortunately, if $K$ is the union of two crumbs at distance slightly larger than $C_2$, it is not necessarily true that $[K]$ is still composed of crumbs (recall that, albeit locally, crumbs can grow under the bootstrap percolation process), which can be disastrous. This is the reason for introducing `modified crumbs' with $C_2'\ll C_2$, so that in the scenario above all connected components of $[K]$ in $\Gamma'$ are `modified crumbs' (there may now be more than two of them). \begin{defn}[Modified crumb]
We define a \emph{modified crumb} by replacing in Definition~\ref{def:crumb} $\Gamma$ by $\Gamma'$ and $C_2$ by $C_2'$.
\end{defn}
In the sequel we will encounter more `modified' notions and constants (like $C_2'$). These will be applied to $K$ equal to the closure $[K']_\partial$ of some $K'$, which is our initial set of infections. Our ultimate goal is to ensure that simply using these modified notions based on (much smaller) modified constants will compensate the closure operation.

\paragraph{Clusters}
We next consider connected components which are not crumbs. Since they can be very large (particularly so if we are working with the closure of a set), we cut them up into (possibly overlapping) pieces termed `clusters', which have bounded size. Roughly speaking, a `cluster' is any `big, but not too big' connected set of infections.
\begin{defn}[Cluster]
\label{def:cluster}
Fix a finite set $K\subset\Lambda\cap\bbZ^2$. Let $\kappa$ be a connected component of $K$ in $\Gamma$ which is not a crumb. We say that a subset $C$ of $\kappa$ is a \emph{cluster} for $K$ if the following conditions hold.
\begin{itemize}
    \item $\diam (C)\le C_3$.
    \item $C$ is connected in $\Gamma$ (i.e. $C$ is a connected component of $C$ in $\Gamma$).
    \item Either $C=\kappa$ or for all $x\in \kappa\setminus C$ and $y\in C$ such that $x\sim y$ in $\Gamma$ we have $\diam(C\cup\{x\})>C_3$.
\end{itemize}
A cluster is called \emph{boundary cluster} if it is at distance at most $C_2$ from $\partial$. For a cluster $C$ we denote by $Q(C)$ the smallest open quadrilateral with sides perpendicular to $\cS$ containing the set $\{x\in\bbR^2 \colon d(x,C)< C_4\}$. 

We similarly define \emph{modified cluster} and \emph{modified boundary cluster} by replacing $\Gamma$ by $\Gamma'$ and $C_2$ by $C_2'$. For a cluster or modified cluster $C$ we denote by $Q'(C)$ the smallest open quadrilateral with sides perpendicular to $\cS$ containing the set $\{x\in\bbR^2 \colon d(x,C)< C_4'\}$.
\end{defn}

\paragraph{Identifying clusters and crumbs}
In order to identify the clusters and crumbs of $K$, one may proceed as follows. Determine the connected components of $K$ in $\Gamma$ and consider each of them separately. For a given component $\kappa$ first check if it is at distance at most $C_2$ from $\partial$. If so, then it is not a crumb and will give rise to clusters. If not, then check if $\kappa$ is the closure of at most $\alpha-1$ sites. If this second verification succeeds, then $\kappa$ is determined to be a crumb and, as mentioned above, it must have diameter much smaller than $C_3$.

If $\kappa$ is thus determined not to be a crumb, we proceed to identify its clusters. If $\diam(\kappa)\le C_3$, then there is a single cluster --- $\kappa$ --- and we are done. If not, we construct the clusters of $\kappa$ by the following algorithm. Initialise the set $C=\varnothing$. If there exists $y\in\kappa\setminus C$ such that $C\cup\{y\}$ is connected in $\Gamma$ and has diameter at most $C_3$, then replace $C$ by $C\cup\{y\}$ and repeat. If several such $y$ exist, then we do this for each possible $y$ in parallel. The clusters containing $x$ are all possible sets $C$ obtained via this algorithm to which no $y$ can be added.

In particular, this provides us with a partition of $K$ into well separated crumbs, single clusters equal to their corresponding connected component and sets of overlapping clusters whose union is a connected component of diameter larger than $C_3$. 

\paragraph{First properties of clusters} Following the algorithm above, we obtain some basic properties of clusters.
\begin{obs}\label{obs:clusters:sites}
Let $C$ be a non-boundary cluster or non-boundary modified cluster for a finite $K\subset\Lambda\cap\bbZ^2$. Then $|C|\ge \alpha$.
\end{obs}
\begin{proof}
Let $\kappa$ be the connected component of $K$ in $\Gamma$ containing $C$. If $\diam(\kappa)\le C_3$, then $C=\kappa$ and $\kappa$ would be a crumb if we had $|\kappa|\le \alpha-1$, by taking $P_\kappa\supset\kappa$. If, on the contrary, $\diam(\kappa)>C_3$, then $\diam(C)\ge C_3-C_2$ (by the third condition of Definition \ref{def:cluster}) and we can choose $C_3$ large enough to have $\frac{C_3-C_2}{C_2}\ge\alpha$.
\end{proof}
Finally, for every cluster $C$ we have $\diam(C)\le C_3$, so $C$ intersects at most $2^{5C^2_3}$ other clusters. Also, $Q(C)\supset[C]$, since $Q(C)\cap\bbZ^2\supset C$ is stable. Furthermore, $\diam(Q(C))=\Theta(C_4)$, as $\diam(C)\le C_3$. Analogous statements hold for modified clusters.

\subsection{Distorted Young diagrams}

\begin{figure}[t]
\floatbox[{\capbeside\thisfloatsetup{capbesideposition={right,top}}}]{figure}[\FBwidth]{
\begin{tikzpicture}[line cap=round,line join=round,,x=0.5cm,y=0.5cm]
\clip(-6,-7.5) rectangle (12.5,14.5);
\fill[fill=black,fill opacity=0.25] (12.15,14.19) -- (-3.57,11.57) -- (-4,10) -- (-3.73,9) -- (-4,8) -- (-3.82,7.33) -- (-5,3) -- (-3.68,-1.83) -- (-4,-3) -- (-3.39,-5.24) -- cycle;
\draw [->] (-1,12) -- (-1.16,12.99);
\draw [->] (4,4) -- (4.78,3.38);
\draw (12.15,14.19)-- (-3.57,11.57);
\draw (-3.57,11.57)-- (-4,10);
\draw (-4,10)-- (-3.73,9);
\draw (-3.73,9)-- (-4,8);
\draw (-4,8)-- (-3.82,7.33);
\draw (-3.82,7.33)-- (-5,3);
\draw (-5,3)-- (-3.68,-1.83);
\draw (-3.68,-1.83)-- (-4,-3);
\draw (-4,-3)-- (-3.39,-5.24);
\draw (-3.39,-5.24)-- (12.15,14.19);
\draw [->] (-4.37,5.31) -- (-5.33,5.57);
\draw [->] (-4.24,0.23) -- (-5.21,-0.03);
\draw (-3.57,11.57)-- (-5.77,3.5);
\draw (-5.77,3.5)-- (-3.39,-5.24);
\draw [very thick] (-3,13.67)-- (-4,10);
\draw [very thick] (-4,10)-- (-3.73,9);
\draw [very thick] (-3.73,9)-- (-4,8);
\draw [very thick] (-4,8)-- (-3.82,7.33);
\draw [very thick] (-3.82,7.33)-- (-5,3);
\draw [very thick] (-5,3)-- (-3.68,-1.83);
\draw [very thick] (-3.68,-1.83)-- (-4,-3);
\draw [very thick] (-4,-3)-- (-3,-6.67);
\draw [very thick] (-3,-6.67)-- (-3,13.67);
\draw (-3,-10) -- (-3,20);
\draw [->] (-3,-7.18) -- (-4,-7.18);
\fill [color=black] (-3.39,-5.24) circle (1.5pt);
\draw[color=black] (-4,-5.24) node {$x_1$};
\fill [color=black] (-3.68,-1.83) circle (1.5pt);
\draw[color=black] (-3.3,-1.5) node {$x_2$};
\fill [color=black] (-3.82,7.33) circle (1.5pt);
\draw[color=black] (-3.4,7) node {$x_3$};
\fill [color=black] (-3.73,9) circle (1.5pt);
\draw[color=black] (-3.4,9.4) node {$x_4$};
\fill [color=black] (-3.57,11.57) circle (1.5pt);
\draw[color=black] (-4,11.57) node {$x_5$};
\fill [color=black] (12.15,14.19) circle (1.5pt);
\draw[color=black] (12.25,13.5) node {$x$};
\fill [color=black] (-5,3) circle (1.5pt);
\draw[color=black] (-4.5,3) node {$y_2$};
\fill [color=black] (-4,-3) circle (1.5pt);
\draw[color=black] (-3.5,-3) node {$y_1$};
\fill [color=black] (-4,8) circle (1.5pt);
\draw[color=black] (-3.5,8) node {$y_3$};
\fill [color=black] (-4,10) circle (1.5pt);
\draw[color=black] (-3.5,10) node {$y_4$};
\fill [color=black] (-5.77,3.5) circle (1.5pt);
\draw[color=black] (-5.4,3.5) node {$y$};
\draw[color=black] (-0.5,13) node {$v_2$};
\draw[color=black] (5,3.6) node {$v_1$};
\draw[color=black] (-5,5.75) node {$u_1$};
\draw[color=black] (-5,0.25) node {$u_2$};
\draw[color=black] (-3.5,-6.85) node {$u'$};
\draw[color=black] (2,7) node {$D$};
\draw[color=black] (5,0.5) node {$\partial$};
\end{tikzpicture}}{
\caption{The shaded region $D$ is a distorted Young diagram (DYD) as in Definition~\ref{def:DYD}. The larger quadrilateral with vertices $x$, $x_1$, $y$ and $x_5$ is $Q(D)$. Note that $Q(D)$ can degenerate into a triangle, but we call it a quadrilateral nevertheless. On the figure $|D|$ is the length of the $v_1$ side, but this is not always the case. The thickened region is the cut distorted Young diagram (CDYD) $C(D)$ of $D$. The vertical line is the boundary between $\Lambda$ on its left and $\partial$ on its right.}
\label{fig:DYD}}
\end{figure}
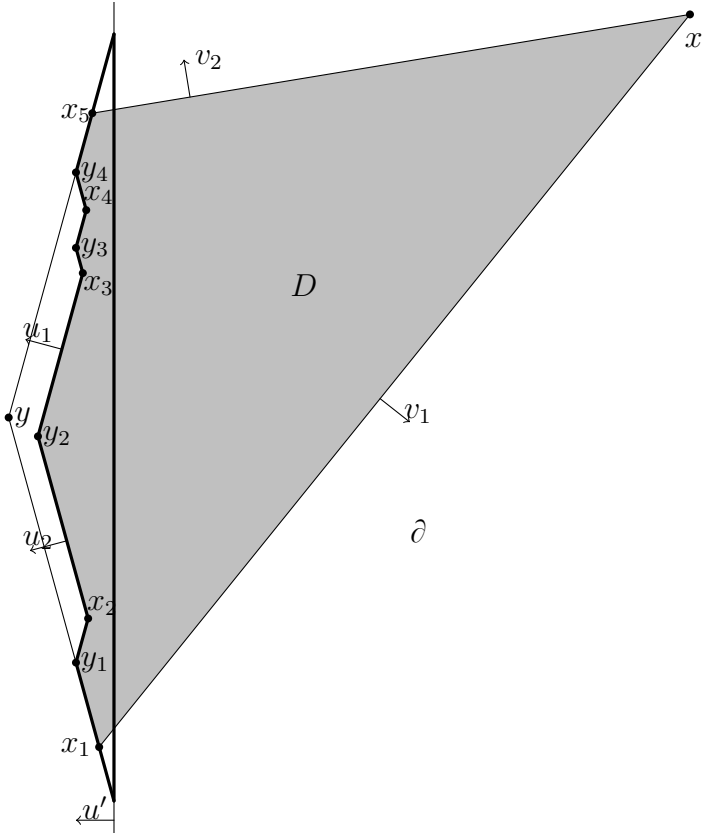

We now define the shape that our `droplets' will have, which resembles Young diagrams\footnote{For the 3-rule model alluded to in Section \ref{sec:sketch} stable sets consist precisely of Young diagrams and the directions $\cS$ provided by Lemma~\ref{lem:directions} can be arbitrarily close to the four axis directions, yielding Young diagrams.}. The following definitions are illustrated in Figure~\ref{fig:DYD}.
\begin{defn}[DYD]
\label{def:DYD}
A \emph{distorted Young diagram} (DYD) is a subset of $\bbR^2$ of the form
\begin{equation}
\label{eq:def:DYD}
(\bbH_{v_1}(x)\cap\bbH_{v_2}(x))\cap \bigcap_{i\in I}(\bbH_{u_1}(x_i)\cup\bbH_{u_2}(x_i))
\end{equation}
for a finite set $I$, some set $X=\{x_i \colon i\in I\}$ of vectors $x_i\in\bbR^2$ and $x\in\bbR^2$. The vectors $x_i$ and $x$ are uniquely defined up to redundancy (and up to the convention that all $x_i$ are on the topological boundary of the DYD). Alternatively, a DYD can also be defined by
\begin{equation}
\label{eq:def:DYD:Y}
(\bbH_{v_1}(x)\cap\bbH_{v_2}(x))\cap \bigcup_{i\in I}(\bbH_{u_1}(y_i)\cap\bbH_{u_2}(y_i)),
\end{equation}
where $y_i$ are the convex corners of the diagram rather than the concave ones.
\end{defn}
For any DYD $D$ we denote by $y$ the vector such that
\[\<y,u_j\>=\sup_{a\in D}\<a,u_j\>=\max_{i\in I}\<y_i,u_j\>\]
for $j\in\{1,2\}$. We further denote
\[Q(D)=\bbH_{u_1}(y)\cap\bbH_{u_2}(y)\cap\bbH_{v_1}(x)\cap\bbH_{v_2}(x),\]
i.e. the minimal open quadrilateral containing $D$ with sides directed by $\cS$. In these terms, for any cluster or modified cluster $C$ we have that $Q(C)$ and $Q'(C)$ are DYD, $Q(Q(C))=Q(C)$ and $Q(Q'(C))=Q'(C)$.

\begin{defn}[CDYD]
\label{def:CDYD}
A \emph{cut distorted Young diagram} (CDYD) is a subset of $\bbR^2$ of the form 
\[\Lambda \cap (\bbH_{u_1}(y)\cap\bbH_{u_2}(y))\cap \bigcap_{i\in I}(\bbH_{u_1}(x_i)\cup\bbH_{u_2}(x_i))\]
for a finite set $I$ and some vectors $x_i\in\bbR^2$ and $y\in\Lambda$. Alternatively, one can write 
\[\Lambda \cap \bigcup_{i\in I}(\bbH_{u_1}(y_i)\cap\bbH_{u_2}(y_i)),\]
where $y_i\in\Lambda$ are the convex corners.
\end{defn}

For a DYD, $D$, we denote by $C(D)$ the CDYD defined by the same $x_i$ and $y$ or the same $y_i$. We extend the notation $C(D)$ to CDYD by setting $C(D)=D$ if $D$ is a CDYD. Note that by Lemma~\ref{lem:directions} all DYD and CDYD are stable for the bootstrap percolation dynamics (restricted to $\Lambda$). Also pay attention to the fact that CDYD are not necessarily connected, contrary to DYD.

\begin{defn}[Size]
For a DYD $D$ we set $\pi(D)=\{x\in\bbR\colon\exists\, y\in D,\<y,v_1+\pi/2\>=x\}$ to be its \emph{projection} (parallel to $v_1$) and $|D|=\sup\pi(D)-\inf\pi(D)$ to be its \emph{size} -- the length of the projection. For a CDYD $D$ we denote its \emph{size} $\diam(D)/C_1$ by $|D|$.
\end{defn}

Note that if $D$ is a DYD, then $|D|=|Q(D)|$ by Lemma~\ref{lem:directions} and the assumption we made that $u_2 < v_1 -\pi/2$. Furthermore, for all DYD $\diam(D)=\Theta(|D|)$ again by Lemma~\ref{lem:directions} with constants depending only on $\cS$. One should be careful with the meaning of size for disconnected CDYD, but it will not cause problems, as all CDYD arising in our forthcoming algorithm are connected.

\begin{obs}
\label{obs:entropy}
Note that for any $d\ge 1$ the number of discretised DYD and CDYD (i.e. intersections of a DYD or CDYD with $\bbZ^2$) containing a fixed point $a\in\bbR^2$ of diameter at most $d$ is less than $c^d$ for some constant $c$ depending only on $\cS$. 
\end{obs}
\begin{proof}
Note that a DYD or CDYD is uniquely determined by its rugged edge formed by its $u_1$ and $u_2$-sides. However, this edge injectively defines an oriented percolation path with directions perpendicular to $u_1$ and $u_2$ on the lattice 
\[\{x\in\bbR^2\colon\exists x_1,x_2\in\bbZ^2,\<x,u_1\>=\<x_1,u_1\>,\<x,u_2\>=\<x_2,u_2\>\}\]
(except its endpoints, which lie on similar lattices). Since the graph-length of this path is bounded by $O(d)$ and its endpoints are within distance $d$ from $a$, the result follows.
\end{proof}

\subsection{Span}\label{sec:span}
\begin{figure}[t]
\resizebox{0.9\textwidth}{!}{
\definecolor{uuuuuu}{rgb}{0,0,0}
\definecolor{ffqqqq}{rgb}{1,0,0}
\definecolor{qqffqq}{rgb}{0,1,0}
\begin{tikzpicture}[line cap=round,line join=round,x=0.9cm,y=0.9cm]
\clip(-9.2,-6.1) rectangle (4.8,7.7);
\fill[color=qqffqq,fill=qqffqq,fill opacity=0.25] (2,3) -- (-3.92,7.44) -- (-4.38,6.83) -- (-3,5) -- (-3.75,4) -- (-3,3) -- (-5,0.33) -- (-4,-1) -- (-4.88,-2.17) -- cycle;
\fill[fill=black,pattern=horizontal lines] (-3.92,7.44) -- (4.71,0.96) -- (-4.37,-5.85) -- (-4.88,-5.17) -- (-4,-4) -- (-5.98,-1.36) -- (-4.86,0.14) -- (-5,0.33) -- (-3.42,2.44) -- (-3.92,3.11) -- (-3.5,3.67) -- (-3.75,4) -- (-3.53,4.3) -- (-3.78,4.63) -- (-3.25,5.33) -- (-4.38,6.83) -- cycle;
\draw [color=qqffqq] (2,3)-- (-3.92,7.44);
\draw [color=qqffqq] (-3.92,7.44)-- (-4.38,6.83);
\draw [color=qqffqq] (-4.38,6.83)-- (-3,5);
\draw [color=qqffqq] (-3,5)-- (-3.75,4);
\draw [color=qqffqq] (-3.75,4)-- (-3,3);
\draw [color=qqffqq] (-3,3)-- (-5,0.33);
\draw [color=qqffqq] (-5,0.33)-- (-4,-1);
\draw [color=qqffqq] (-4,-1)-- (-4.88,-2.17);
\draw [color=qqffqq] (-4.88,-2.17)-- (2,3);
\draw [line width=2pt,color=ffqqqq] (2.91,-0.39)-- (-4.37,-5.85);
\draw [line width=2pt,color=ffqqqq] (-4.37,-5.85)-- (-4.88,-5.17);
\draw [line width=2pt,color=ffqqqq] (-4.88,-5.17)-- (-4,-4);
\draw [line width=2pt,color=ffqqqq] (-4,-4)-- (-5.98,-1.36);
\draw [line width=2pt,color=ffqqqq] (-5.98,-1.36)-- (-3.28,2.25);
\draw [line width=2pt,color=ffqqqq] (-3.28,2.25)-- (-3.92,3.11);
\draw [line width=2pt,color=ffqqqq] (-3.92,3.11)-- (-3.28,3.97);
\draw [line width=2pt,color=ffqqqq] (-3.28,3.97)-- (-3.78,4.63);
\draw [line width=2pt,color=ffqqqq] (-3.78,4.63)-- (2.91,-0.39);
\draw (-3.92,7.44)-- (4.71,0.96);
\draw (4.71,0.96)-- (-4.37,-5.85);
\draw (-4.37,-5.85)-- (-4.88,-5.17);
\draw (-4.88,-5.17)-- (-4,-4);
\draw (-4,-4)-- (-5.98,-1.36);
\draw (-5.98,-1.36)-- (-4.86,0.14);
\draw (-4.86,0.14)-- (-5,0.33);
\draw (-5,0.33)-- (-3.42,2.44);
\draw (-3.42,2.44)-- (-3.92,3.11);
\draw (-3.92,3.11)-- (-3.5,3.67);
\draw (-3.5,3.67)-- (-3.75,4);
\draw (-3.75,4)-- (-3.53,4.3);
\draw (-3.53,4.3)-- (-3.78,4.63);
\draw (-3.78,4.63)-- (-3.25,5.33);
\draw (-3.25,5.33)-- (-4.38,6.83);
\draw (-4.38,6.83)-- (-3.92,7.44);
\draw [dash pattern=on 4pt off 4pt,color=uuuuuu] (-9.13,0.5)-- (-3.92,7.44);
\draw [dash pattern=on 4pt off 4pt,color=uuuuuu] (-9.13,0.5)-- (-4.37,-5.85);
\draw [dash pattern=on 4pt off 4pt,color=qqffqq] (-8,2)-- (-3.92,7.44);
\draw [dash pattern=on 4pt off 4pt,color=qqffqq] (-8,2)-- (-4.88,-2.17);
\draw [line width=2pt,dash pattern=on 4pt off 4pt,color=ffqqqq] (-8,-1)-- (-4.37,-5.85);
\draw [line width=2pt,dash pattern=on 4pt off 4pt,color=ffqqqq] (-8,-1)-- (-3.78,4.63);
\begin{scriptsize}
\fill [color=qqffqq] (-4,-1) circle (1.5pt);
\draw (-3.87,-0.75) node {$x_1^1$};
\fill [color=qqffqq] (-3,3) circle (1.5pt);
\draw (-2.7,3) node {$x_2^1$};
\fill [color=qqffqq] (-3,5) circle (1.5pt);
\draw (-2.8,5) node {$x_3^1$};
\fill [color=qqffqq] (-3.92,7.44) circle (1.5pt);
\draw (-4.7,7.44) node {$x_4^1=x_8$};
\fill [color=qqffqq] (-4.88,-2.17) circle (1.5pt);
\draw (-4.5,-2.17) node {$y_1^1$};
\fill [color=qqffqq] (-5,0.33) circle (1.5pt);
\draw (-5.7,0.5) node {$y_2^1=y_3$};
\fill [color=qqffqq] (-3.75,4) circle (1.5pt);
\draw (-4.45,4) node {$y_3^1=y_5$};
\fill [color=qqffqq] (-4.38,6.83) circle (1.5pt);
\draw (-5.1,6.83) node {$y_4^1=y_7$};
\fill [color=qqffqq] (-8,2) circle (1.5pt);
\draw (-7.7,2) node {$y^1$};
\fill [color=qqffqq] (2,3) circle (1.5pt);
\draw (2.13,3.24) node {$x^1$};
\fill [color=ffqqqq] (-3.78,4.63) circle (1.5pt);
\draw (-4.5,4.63) node {$y_4^2=y_6$};
\fill [color=ffqqqq] (-3.92,3.11) circle (1.5pt);
\draw (-4.45,2.9) node {$y_3^2=y_4$};
\fill [color=ffqqqq] (-5.98,-1.36) circle (1.5pt);
\draw (-6.7,-1.36) node {$y_2^2=y_2$};
\fill [color=ffqqqq] (-4.88,-5.17) circle (1.5pt);
\draw (-5.6,-5.17) node {$y_1^2=y_1$};
\fill [color=ffqqqq] (-3.28,3.97) circle (1.5pt);
\draw (-3.1,3.8) node {$x_4^2$};
\fill [color=ffqqqq] (-3.28,2.25) circle (1.5pt);
\draw (-3,2.25) node {$x_3^2$};
\fill [color=ffqqqq] (-4,-4) circle (1.5pt);
\draw (-4.8,-4) node {$x_2^2=x_2$};
\fill [color=ffqqqq] (-4.37,-5.85) circle (1.5pt);
\draw (-3.5,-5.85) node {$x_1^2=x_1$};
\fill [color=ffqqqq] (-8,-1) circle (1.5pt);
\draw (-7.6,-1) node {$y^2$};
\fill [color=ffqqqq] (2.91,-0.39) circle (1.5pt);
\draw (3,-0.5) node {$x^2$};
\fill (4.71,0.96) circle (1.5pt);
\draw (4.7,1.21) node {$x$};
\fill (-9.13,0.5) circle (1.5pt);
\draw (-8.9,0.5) node {$y$};

\fill (-3.92,7.44) circle (1.5pt);
\fill (-3.25,5.33) circle (1.5pt);
\draw (-3.6,5.33) node {$x_7$};
\fill (-3.53,4.3) circle (1.5pt);
\draw (-3.8,4.3) node {$x_6$};
\fill (-3.5,3.67) circle (1.5pt);
\draw (-3.9,3.67) node {$x_5$};
\fill (-3.42,2.44) circle (1.5pt);
\draw (-3.8,2.44) node {$x_4$};
\fill (-4.86,0.14) circle (1.5pt);
\draw (-5.2,0.14) node {$x_3$};
\fill (-4,-4) circle (1.5pt);
\fill (-4.37,-5.85) circle (1.5pt);

\fill (-4.38,6.83) circle (1.5pt);
\fill (-3.78,4.63) circle (1.5pt);
\fill (-3.75,4) circle (1.5pt);
\fill (-3.92,3.11) circle (1.5pt);
\fill (-5,0.33) circle (1.5pt);
\fill (-5.98,-1.36) circle (1.5pt);
\fill (-4.88,-5.17) circle (1.5pt);

\draw (-1,4) node {$D_1$};
\draw (-1,-1) node {$D_2$};
\draw (2.5,1) node {$D_1\vee D_2$};
\end{scriptsize}
\end{tikzpicture}}
\caption{The shaded region $D_1$ and thickened region $D_2$ are DYD.  Their respective quadrilaterals $Q(D_i)$ are completed by dashed lines. Their span $D_1\vee D_2$ is hatched and its quadrilateral $Q(D_1\vee D_2)$ is also completed by dashed lines.}
\label{fig:span}
\end{figure}
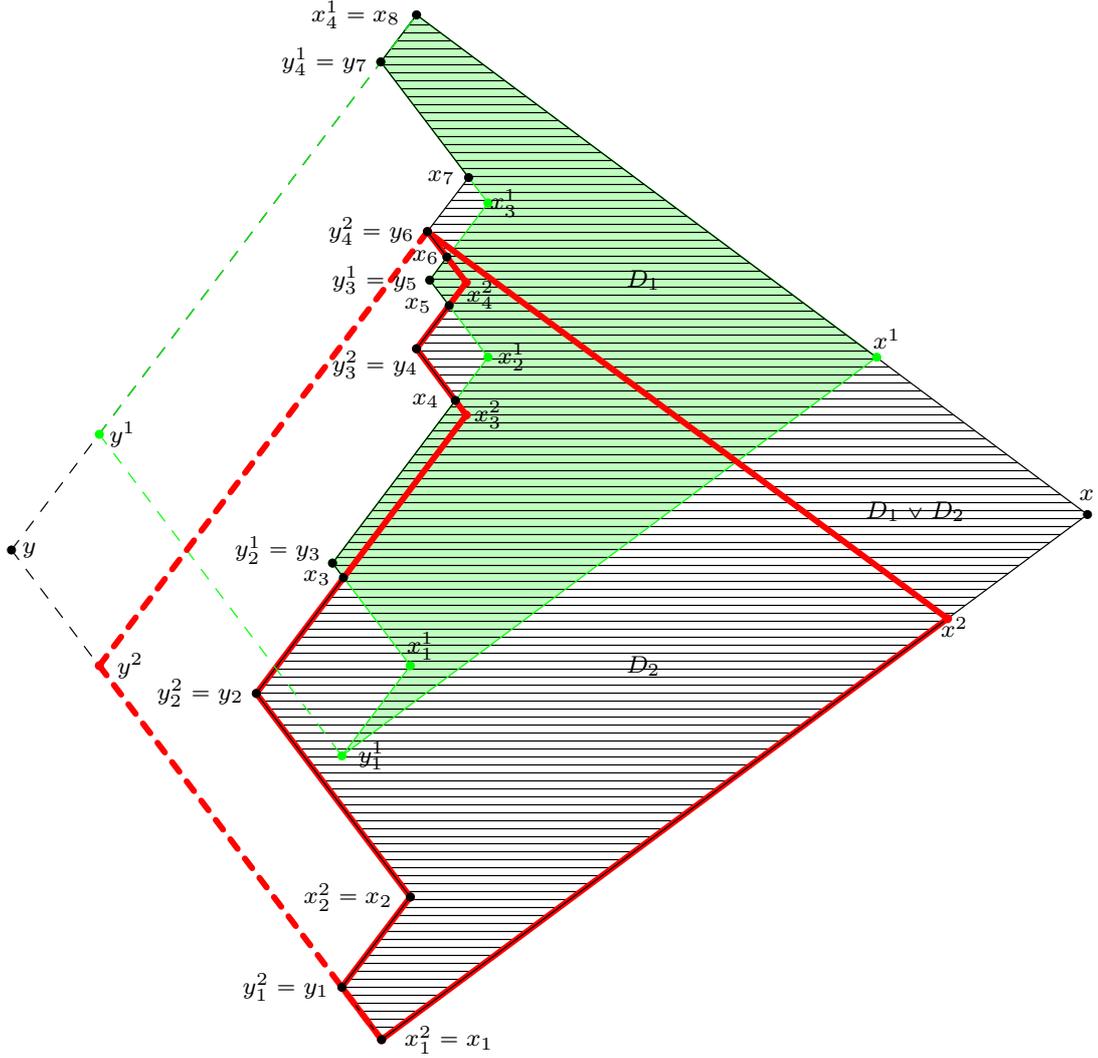

We next introduce a procedure of merging DYD and CDYD. This will be used only for couples of intersecting ones, but can be defined regardless of whether they intersect. The operation is illustrated in Figure~\ref{fig:span}.
\begin{lem}
\label{lem:span}
For any two DYD, $D_1$ and $D_2$, the minimal DYD containing $D_1\cup D_2$ is well defined. We denote it by $D_1\vee D_2$ and call it their \emph{span}. The operation $\vee$ is associative\footnote{Associativity was referred to as commutativity by previous authors~\cite{Bollobas15}.} and commutative.
\end{lem}
\begin{proof}
Let $D_1$ be defined by $Y^1=\{y_i^{1} \colon i\in I\},x^1$ (see~\eqref{eq:def:DYD:Y}) and similarly for $D_2$. Let $x\in\bbR^2$ be the vector such that $\bbH_{v_i}(x^1)\cup\bbH_{v_i}(x^2)=\bbH_{v_i}(x)$ for $i\in\{1,2\}$. Let $Y$ be the set of $y_i\in Y^1\cup Y^2$ such that for all $y_j\in Y^1\cup Y^2$ with $y_i\neq y_j$ we have $\bbH_{u_1}(y_j)\cap\bbH_{u_2}(y_j)\not{\supset}\bbH_{u_1}(y_i)\cap\bbH_{u_2}(y_i)$. We denote by $D$ the DYD defined by $Y,x$ and claim that for any DYD $D'\supset D_1\cup D_2$ we have $D'\supset D$, which is enough to conclude that $D=D_1\vee D_2$ is well defined. Let $D'$ be defined by $Y',x'$.

Note that for each $y_i\in Y$ (and in fact in $Y_1\cup Y_2$) there is a sequence of points in $D_1$ or $D_2$ converging to $y_i$, so that (by extraction of a subsequence) there exists $y'_j$ with $\bbH_{u_1}(y'_j)\cap\bbH_{u_2}(y'_j)\supset\bbH_{u_1}(y_i)\cap\bbH_{u_2}(y_i)$. Similarly, there is a sequence of points in $D_1$ or $D_2$ converging to the boundary of $\bbH_{v_1}(x)$, so that $\bbH_{v_1}(x')\supset \bbH_{v_1}(x)$ and similarly for $v_2$. Thus, we do have $D'\supset D$.

Finally, the commutativity is obvious and the associativity follows from the characterisation of $D_1\vee D_2$ as the minimal DYD containing both $D_1$ and $D_2$.
\end{proof}

We analogously define the \emph{span} $D_1\vee D_2$ of two CDYD $D_1$ and $D_2$ -- the minimal CDYD containing both -- and note that it coincides with their union (which is also commutative and associative). We also define the \emph{span} $C\vee D$ of a DYD $D$ and a CDYD $C$ as the minimal CDYD containing $(C\cup D)\setminus\partial$, which coincides with $C\vee C(D)$. The proof that it is well defined is analogous to Lemma~\ref{lem:span}. 

We have thus defined an associative and commutative binary operation $\vee$ on all DYD and CDYD. Moreover, the idempotent unary operation $C(\cdot)$ is distributive with respect to $\vee$ and $C(D_1)\vee D_2=C(D_1\vee D_2)$. Furthermore, the span of several DYD is the minimal DYD containing all of them, while the span of several DYD and at least one CDYD is the minimal CDYD containing all the corresponding CDYD.

\subsection{Droplet algorithm and spanned droplets}

A \emph{droplet} is any DYD contained in $\Lambda$ or CDYD.
We are now ready to  define our
\emph{droplet algorithm}, which takes as input a finite set $K\subset\Lambda\cap\bbZ^2$ of infections and outputs a set $\cD$ of disjoint connected droplets. It proceeds as follows.
\begin{itemize}
    \item Form an initial collection of DYD $\cD$ consisting of $Q(C)$ for all clusters $C$ of $K$. If a DYD $D\in\cD$ intersects $\partial$, replace it by its CDYD, $C(D)$, to obtain a droplet.
    \item As long as it is possible, replace two intersecting droplets of $\cD$ by their span. If the span intersects $\partial$, replace it by its CDYD to obtain a droplet.
    \item Output the collection $\cD$ obtained when all droplets are disjoint.
\end{itemize}
We similarly define the \emph{modified droplet algorithm} by replacing $Q(C)$ by $Q'(C)$ and clusters by modified clusters above.

The output $\cD$ is clearly a collection of disjoint connected droplets. Indeed, by induction all $x_i$ corners of droplets remain in $\Lambda$ (see Figure~\ref{fig:span}), so that DYD remain connected when replaced by CDYD.

\begin{rem}
\label{rem:algo}
From the results of Section \ref{sec:span} it is clear that the order of merging does not impact the output of the algorithm, which is thus well defined. It can also be expressed as the minimal collection of disjoint droplets containing the intersection with $\Lambda$ of the original collection of quadrilaterals. This minimal collection is well defined. Consequently, the union of the output is increasing in the input.
\end{rem}

\begin{defn}[Spanned droplets]
\label{def:spanned}
Let $D$ be a droplet and $K\subset\bbZ^2$. We say that $D$ is \emph{spanned} for $K$ with boundary $\partial$ if the output of the droplet algorithm for $K\cap D$ has a droplet containing $D$. We omit $K$ and $\partial$ if they are clear from the context. Similarly, $D$ is \emph{modified spanned} if the output of the modified droplet algorithm for $K\cap D$ has a droplet containing $D$.
\end{defn}
Note that, when seen as an event, a droplet being spanned is monotone. It is also clear that each droplet appearing in (the intermediate or final stages of) the droplet algorithm is spanned and similarly for the modified droplet algorithm. Indeed, the clusters responsible for creating a droplet in the course of the algorithm are contained in the droplet, so each of them is still a cluster of $K\cap D$ (recall that crumbs  have diameter much smaller than $C_3$).

\subsection{Properties of the algorithm}
We next establish several properties of the algorithm. The approach is similar to the one of~\cite{Bollobas14} with the notable exception of the key Closure Proposition~\ref{prop:closure}. We start with the following purely geometric statement.
\begin{lem}[Subadditivity]
\label{lem:subadd:diam}
Let $D_1$ and $D_2$ be two DYD or CDYD with non-empty intersection. Then
\[|D_1\vee D_2|\le |D_1|+|D_2|.\]
Furthermore, if $D$ is a DYD intersecting $\partial$, then $|C(D)|\le|D|$.
\end{lem}
\begin{proof}
First assume that $D_1$ and $D_2$ are DYD. Since $|D|=|Q(D)|$ for any DYD $D$ and $D_1\vee D_2\subset Q(Q(D_1)\vee Q(D_2))$, it suffices to prove the assertion for merging quadrilaterals instead of DYD. But in that case it is not hard to check directly and is a particular case of Lemma~15 of the first arXiv version of~\cite{Bollobas15} (or Lemma~23 of the second version). Since similar (but actually slightly more involved) details were omitted in the proof of the corresponding Lemma~4.6 of~\cite{Bollobas15} and differed to earlier versions, we will not go into useless detail here either. To give a sketch of a possible argument, one can check that for fixed shapes of $Q(D_1)$ and $Q(D_2)$ the maximal $Q(Q(D_1)\vee Q(D_2))$ is achieved when their intersection is reduced to a vertex. Yet, in those configurations one can obtain the $v_1$ and $v_2$ sides of $Q(Q(D_1)\vee Q(D_2))$ as the union of those of $Q(D_1)$ and translates of those of $Q(D_2)$ (see Figure~\ref{fig:span}). This concludes the proof, as only $v_1$ and (possibly) $v_2$ sides contribute to $|\cdot|$ by Lemma~\ref{lem:directions}.

Next assume that $D_1$ is a DYD and $D_2$ is a CDYD. Let $Y=\{y_i \colon i\in I\}$ be the set of vectors defining $C(D_1)$ and let $a\in D_1\cap D_2$. Since $Y\subset \overline{D_1}$, we have that $d(y_i,a)\le \diam (D_1)$. It then easily follows that the CDYD defined by only one corner, $y_i$, which we denote $C(y_i)$, is within distance $O(\diam(D_1))$ from $C(a)$. But then $C(D_1)=\bigcup_{i\in I} C(y_i)$ is within distance $O(\diam(D_1))$ from $C(a)$. Thus, $|D_1\vee D_2|\le (\diam (D_2)+O(\diam (D_1)))/C_1\le |D_2| + |D_1|$, since $\diam(D_1)=O(|D_1|)$ and all implicit constants depend only on $\cS$ and are thus much smaller than $C_1$.

Next assume that $D_1$ and $D_2$ are CDYD. Then the statement is trivial, because $D_1\vee D_2=D_1\cup D_2$, so $\diam (D_1)+\diam(D_2)\ge \diam(D_1\vee D_2)$ by the triangle inequality.

Finally, let $D$ be a DYD intersecting $\partial$. Then, $|C(Q(D))|\ge|C(D)|$ and $|Q(D)|=|D|$, so we may assume that $D=Q(D)$ and prove $|C(D)|\le|D|$. But in this case it is easy to see that $\diam (C(D))=O(\diam(D))=O(|D|)$ with constants depending only on $\cS$, which concludes the proof.
\end{proof}

The subadditivity lemma will be used to prove the next two adaptations of classical results.
\begin{lem}[Aizenman-Lebowitz]
\label{lem:AL}
Let $K$ be a finite set and let $D$ be a spanned droplet with $|D|\ge C_4^2$. Then for all $C_4^2/C_1\le k\le |D|/C_1$ there exists a connected spanned droplet $D'$ with $k\le |D'|\le 2k$. The same statement holds for modified spanned droplets.
\end{lem}
\begin{proof}
By Lemma~\ref{lem:subadd:diam} at each step of the droplet algorithm the largest size of a droplet appearing in the collection at most doubles. Initially the largest size is at most $C_1C_4$ and in the end there is a (unique) droplet $D''\supset D$, so that $|D''|\ge |D|/C_1\ge C_4^2/C_1>C_1C_4$. Then there is a stage of the algorithm at which the maximal size of a droplet in $\cD$ is between $k$ and $2k$, which is enough since all droplets appearing in the droplet algorithm are connected and spanned. The proof for modified spanned droplets is identical, using the modified droplet algorithm.
\end{proof}

\begin{lem}[Extremal]
\label{lem:extremal}
Let $K \subset \bbZ^2$ and let $D$ be a droplet spanned for $K$. Then the total number of disjoint clusters for $K \cap D$ in $D$ is at least $\diam(D)/C_4^2$.
\end{lem}
\begin{proof}
In this proof all clusters will be clusters for $K \cap D$. Assume that at the initial stage of the algorithm there are $k$ clusters (not disjoint). One can then find $k/C_4'$ disjoint ones, since their diameter is at most $C_3$. Furthermore, by Lemma~\ref{lem:subadd:diam} the total size of droplets in the collection $\cD$ is decreasing, so that $|D|/C_1\le|D'|\le k C_1C_4$, where $D'\supset D$ is some droplet in the output of the algorithm. Indeed, $|Q(C)|\le C_1C_4$ for all clusters $C$. This concludes the proof, since $|D|\ge \diam(D)/C_1$ for all DYD and CDYD.
\end{proof}

We next transform this extremal bound into an exponential decay of the probability that a droplet is spanned until saturation at the critical size. In the following lemma, we identify the configuration $\omega$ having law $\mu$ and the set of its zeroes.
\begin{lem}[Exponential decay]
\label{lem:exp:decay:diam}
Let $D$ be a droplet with $|D|\le 2/(C_5q^{\alpha})$. Then
\[\mu(D\text{ is spanned for $\omega$})<\exp(-C_4|D|).\]
\end{lem} 
\begin{proof}
Let $D$ be a droplet with $|D|\le 2/(C_5 q^{\alpha})$, so that $\diam(D)=d\le 2C_1/(C_5 q^\alpha)$. By Lemma~\ref{lem:extremal} if $D$ is spanned for $\omega$, it contains at least $d/C_4^2$ disjoint clusters for $\omega \cap D$, each one having diameter at most $C_3$. Each non-boundary cluster has at least $\alpha$ sites by Observation \ref{obs:clusters:sites}, while boundary clusters are non-empty and located at distance at most $C_2$ from $\partial$. Thus, we have the union bound
\begin{align*}
\mu(D\text{ is spanned for $\omega$})\le& \sum_{l=0}^{d/C_4^2}\binom{C_3^{2\alpha} d^2}{l} \binom{C_3d}{d/C_4^2-l} q^{l\alpha +(d/C_4^2-l)}\\
\le &\sum_{l=d/(2C_4^2)}^{d/C_4^2}(C'_4q^\alpha d^2/l)^l.e^d+\sum_{l'=d/(2C_4^2)}^{d/C_4^2}(C'_4qd/l' )^{l'}.e^d\\
\le&\sum_{l=d/(2C_4^2)}^{d/C_4^2}\left(\frac{C_4'e^{2C^2_4}q^\alpha}{1/(2C_4^2)}\cdot\frac{2C_1}{C_5q^\alpha}\right)^{l}+\sum_{l'=d/(2C_4^2)}^{d/C_4^2}\left(2C_4^2C_4'e^{2C_4^2}q\right)^{l'}\\
\le& \exp(-C_4 d),
\end{align*}
recalling that $C_5$ is sufficiently large depending on $C_4$, $C_4'$ and $C_1$.
\end{proof}

Our next aim is to prove that the closure of a set is contained in its droplet collection up to very local infections next to initial ones. To that end we will need some preliminary results, similar to those used by Bollobás, Duminil-Copin, Morris and Smith~\cite{Bollobas14}.

\begin{obs}[Lemma~6.5 of~\cite{Bollobas14}]
\label{obs:C1:well:defined}
Let $u$ be a rational non-semi-isolated stable direction. Let $K\subset\bbZ^2$ with $|K|<\alpha(u)$ (if $\alpha(u)=\infty$ the condition is that $K$ is finite, but there is no a priori bound on its size). Then there exists a constant $C(\cU,u,|K|)$ not depending on $K$ such that $[K]_{\bbH_u}$ is within distance $C(\cU,u,|K|)$ from $K$.
\end{obs}

Since we will require some improvements later, we spell out a  proof of the above result for completeness (actually our proof is slightly different from the one in \cite{Bollobas14}).
\begin{proof}[Proof of Observation \ref{obs:C1:well:defined}]
We prove the statement by induction on $|K|$. For a $K=\{x\}$ this is easy, since if $\<x,u\>$ is sufficiently large $[K]_{\bbH_u}=K$ and otherwise there is a single possible configuration for each value of $\<x,u\>$ up to translation. Assume the result holds for $|K|<n$. If one can write $K=K_1\sqcup K_2$ with $K_1,K_2\neq\varnothing$ and $d(K_1,K_2)>2C(\cU,u,n-1)+O(1)$, then $[K]_{\bbH_u}=[K_1]_{\bbH_u}\sqcup [K_2]_{\bbH_u}$, since $[K_1]_{\bbH_u}$ and $[K_2]_{\bbH_u}$ are at sufficiently large distance, hence no site can use both to become infected. Assume that, on the contrary, there are no large gaps between parts of $K$. There is a finite number of such $K$ up to translation and for each of these $[K]$ is finite (e.g. since $K$ is contained in a quadrilateral with sides perpendicular to $\cS$), so within uniformly bounded distance from $K$. Therefore, if $\bbH_u$ is sufficiently far from $K$, $[K]_{\bbH_u}=[K]$. Otherwise, there is a finite number of possible $K$ up to translation perpendicular to $u$ and for each of them $[K]_{\bbH_u}$ is finite, so that one can indeed find a finite uniform constant $C(\cU,u,n)$ as claimed.
\end{proof}
A quantitative version of this result was proved by Mezei and the first author~\cite{Hartarsky18NP}. An easy corollary of Observation~\ref{obs:C1:well:defined} is the fact that crumbs can only grow very locally (see Figure~\ref{fig:crumbs}). 
\begin{cor}
\label{cor:crumbs:growth}
Let $C_1$ be sufficiently large depending on $\cU$. Let $K\subset\bbZ^2$ with $|K|<\alpha$. Then $[K]$ is within distance $C_1/(6\alpha)$ from $K$. Also, for a (modified) crumb $\kappa$ we have that $\diam([\kappa])\le \alpha C_2$ and $[\kappa]$ is within distance $C_1$ from $\kappa$.
\end{cor}
\begin{proof}
The first assertion follows from Observation~\ref{obs:C1:well:defined}, since if it were wrong, one could simply translate a set $K$ sufficiently far from a half-plane yielding a contradiction with the observation.

Next consider a (modified) crumb $\kappa$ and $P_\kappa$ minimal with $|P_\kappa|<\alpha$ and $[P_\kappa]\supset \kappa$. Then $[\kappa]\subset[P_\kappa]$ is within distance $C_1/(6\alpha)$ from $P_\kappa$. If the sites of $P_\kappa$ are not connected in the graph $\Gamma''$ on $\bbZ^2$ with connections at distance at most $C_1+C_2$, then either $\kappa$ is not connected in $\Gamma$ or $P_\kappa$ is not minimal, which are both contradictions. Similarly, if there is no site of $\kappa$ at distance smaller than $C_1/(2\alpha)$ from a $C_1/(2\alpha)$-connected component of $P_\kappa$, that component can be removed from $P_\kappa$, contradicting minimality. Hence, $P_\kappa$ is within distance $C_1/2$ from $\kappa$. The result is then immediate, as $[\kappa]$ is within distance $C_1/2+C_1/(6\alpha)$ from $\kappa$ and its diameter is at most $C_1/(3\alpha) +\diam(P_\kappa)$, while $\diam(P_\kappa)\le (\alpha-1)(C_1+C_2)$. 
\end{proof}

In order to treat infection at the concave corners of droplets we will need the following modification of Observation~\ref{obs:C1:well:defined}.
\begin{cor}
\label{cor:C1:well:defined}
Let $u_1$ and $u_2$ be rational strongly stable directions such that $\bbH_{u_1}\cup\bbH_{u_2}$ is stable for the bootstrap percolation dynamics i.e. $\nexists U\in\cU, U\subset\bbH_{u_1}\cup\bbH_{u_2}$. Let $K\subset\bbZ^2$ with $|K|\le \alpha-1$. Then $[K]_{\bbH_{u_1}\cup\bbH_{u_2}}$ is within distance $C(\cU,u_1,u_2)$ from $K$.
\end{cor}
\begin{proof}
We apply a similar induction to the one in the proof of Observation~\ref{obs:C1:well:defined}. The only difference is that we can no longer use translation invariance. If $d(K,\bbH_{u_2})>C(\cU,u_1,|K|)+O(1)$, by Observation~\ref{obs:C1:well:defined}, we have $[K]_{\bbH_{u_1}\cup\bbH_{u_2}}=[K]_{\bbH_{u_1}}$ and similarly for $u_1$ and $u_2$ interchanged. We can thus assume that $K$ is within distance $C'(\cU,u_1,u_2)$ from the origin. But then $[K\cup \bbH_{u_1}\cup\bbH_{u_2}]\subset\bbH_{u_1}\cup\bbH_{u_2}\cup\bbH_{u'}(C''(\cU,u_1,u_2)u')$, where $u'=(u_1+u_2)/2$, since the latter region is stable by the hypothesis on $u_1,u_2$.
\end{proof}

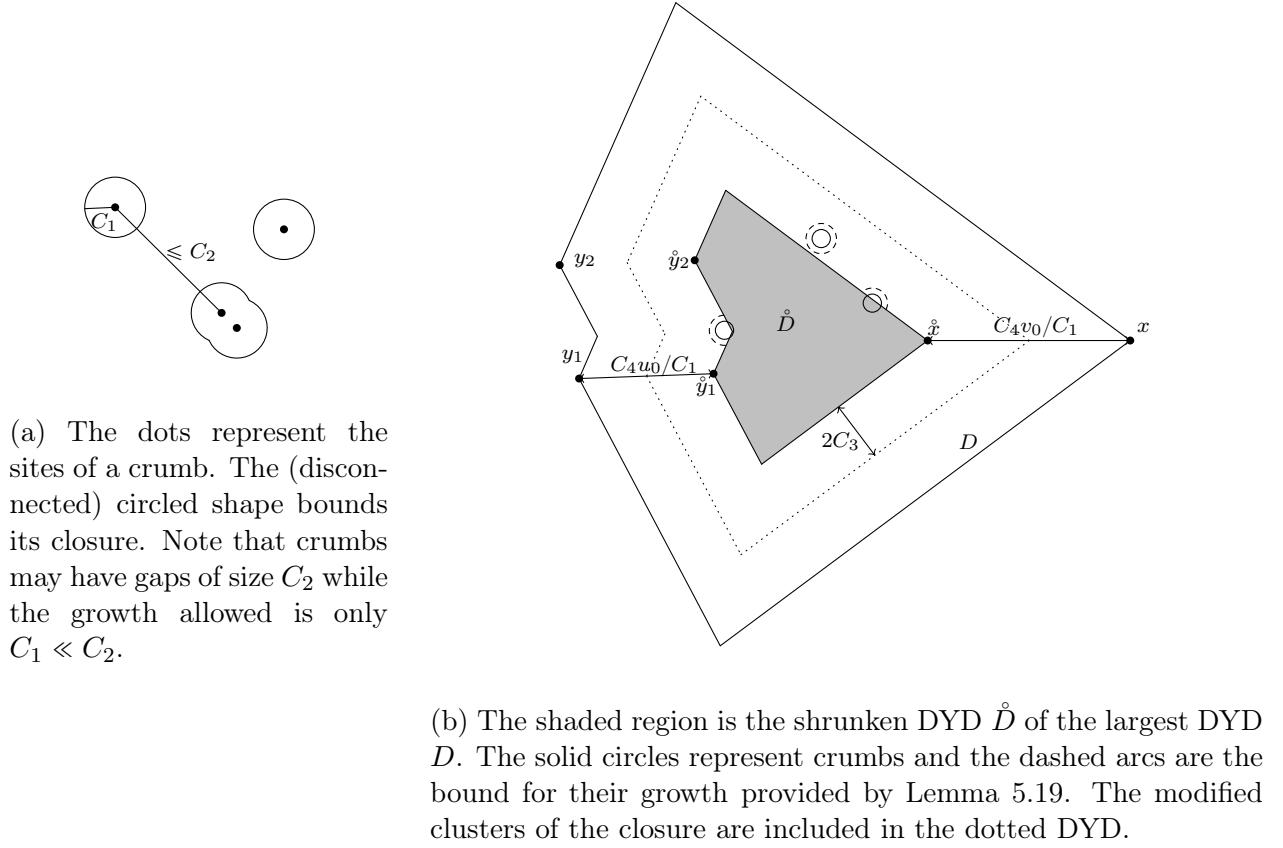
\begin{figure}[t]
\begin{center}
\begin{subfigure}{0.3\textwidth}
\begin{center}
\begin{tikzpicture}[line cap=round,line join=round,x=0.2cm,y=0.2cm]
\clip(-10.5,-2.5) rectangle (5.5,10.5);
\draw(3.1,6.54) circle (0.4cm);
\draw(-8,8) circle (0.4cm);
\draw (-10,7.91)-- (-8,8);
\draw (-8,8)-- (-1,1);
\draw [shift={(-1,1)}] plot[domain=0.42:4.29,variable=\t]({1*2*cos(\t r)+0*2*sin(\t r)},{0*2*cos(\t r)+1*2*sin(\t r)});
\draw [shift={(0,0)}] plot[domain=-2.72:1.15,variable=\t]({1*2*cos(\t r)+0*2*sin(\t r)},{0*2*cos(\t r)+1*2*sin(\t r)});
\begin{scriptsize}
\fill [color=black] (0,0) circle (1.5pt);
\fill [color=black] (-1,1) circle (1.5pt);
\fill [color=black] (3.1,6.54) circle (1.5pt);
\fill [color=black] (-8,8) circle (1.5pt);
\draw[color=black] (-8.7,7) node {$C_1$};
\draw[color=black] (-3,5) node {$\le C_2$};
\end{scriptsize}
\end{tikzpicture}
\end{center}
\subcaption{The dots represent the sites of a crumb. The (disconnected) circled shape bounds its closure. Note that crumbs may have gaps of size $C_2$ while the growth allowed is only $C_1\ll C_2$.}
\label{fig:crumbs}
\end{subfigure}
\quad
\begin{subfigure}{0.66\textwidth}
\begin{center}
\begin{tikzpicture}[line cap=round,line join=round,x=0.8cm,y=0.8cm]
\clip(-5,-5.5) rectangle (5,5.5);

\fill[fill=black,fill opacity=0.25] (-1.99,2.16) -- (-2.5,1) -- (-1.88,-0.18) -- (-2.19,-0.88) -- (-1.4,-2.38) -- (1.33,-0.33) -- cycle;
\draw (-1.99,2.16)-- (-2.5,1);
\draw (-2.5,1)-- (-1.88,-0.18);
\draw (-1.88,-0.18)-- (-2.19,-0.88);
\draw (-2.19,-0.88)-- (-1.4,-2.38);
\draw (-1.4,-2.38)-- (1.33,-0.33);
\draw (1.33,-0.33)-- (-1.99,2.16);
\draw(-2.01,-0.16) circle (0.12cm);
\draw [shift={(-2.01,-0.16)},dash pattern=on 2pt off 2pt]  plot[domain=1.61:4.84,variable=\t]({1*0.25*cos(\t r)+0*0.25*sin(\t r)},{0*0.25*cos(\t r)+1*0.25*sin(\t r)});
\draw(-0.42,1.36) circle (0.12cm);
\draw [dash pattern=on 2pt off 2pt] (-0.42,1.36) circle (0.2cm);
\draw(0.42,0.29) circle (0.12cm);
\draw [shift={(0.42,0.29)},dash pattern=on 2pt off 2pt]  plot[domain=-0.45:2.3,variable=\t]({1*0.25*cos(\t r)+0*0.25*sin(\t r)},{0*0.25*cos(\t r)+1*0.25*sin(\t r)});
\draw (4.66,-0.33)-- (-2.81,5.27);
\draw (-2.81,5.27)-- (-4.72,0.92);
\draw (-4.72,0.92)-- (-4.1,-0.26);
\draw (-4.1,-0.26)-- (-4.4,-0.96);
\draw (-4.4,-0.96)-- (-2.08,-5.39);
\draw (-2.08,-5.39)-- (4.66,-0.33);
\draw [dotted] (3,-0.33)-- (-2.4,3.72);
\draw [dotted] (-2.4,3.72)-- (-3.61,0.96);
\draw [dotted] (-3.61,0.96)-- (-2.99,-0.22);
\draw [dotted] (-2.99,-0.22)-- (-3.29,-0.92);
\draw [dotted] (-3.29,-0.92)-- (-1.74,-3.88);
\draw [dotted] (-1.74,-3.88)-- (3,-0.33);
\draw [<->] (-4.4,-0.96) -- (-2.19,-0.88);
\draw [<->] (4.66,-0.33) -- (1.33,-0.33);
\draw [<->] (-0.14,-1.43) -- (0.46,-2.23);
\begin{scriptsize}
\fill [color=black] (-2.19,-0.88) circle (1.5pt);
\draw[color=black] (-2.3,-1.1) node {$\ring y_1$};
\fill [color=black] (-2.5,1) circle (1.5pt);
\draw[color=black] (-2.75,1) node {$\ring y_2$};
\fill [color=black] (1.33,-0.33) circle (1.5pt);
\draw[color=black] (1.43,-0.12) node {$\ring x$};
\fill [color=black] (4.66,-0.33) circle (1.5pt);
\draw[color=black] (4.88,-0.12) node {$x$};
\fill [color=black] (-4.4,-0.96) circle (1.5pt);
\draw[color=black] (-4.5,-0.6) node {$y_1$};
\fill [color=black] (-4.72,0.92) circle (1.5pt);
\draw[color=black] (-4.3,1) node {$y_2$};
\draw[color=black] (-3.16,-0.74) node {$C_4 u_0/C_1$};
\draw[color=black] (3.12,-0.1) node {$C_4 v_0/C_1$};
\draw[color=black] (-0.1,-2) node {$2C_3$};
\draw[color=black] (-1,0) node {$\ring D$};
\draw[color=black] (2,-2) node {$D$};
\end{scriptsize}
\end{tikzpicture}
\end{center}
\subcaption{The shaded region is the shrunken DYD $\ring D$ of the largest DYD $D$. The solid circles represent crumbs and the dashed arcs are the bound for their growth provided by Lemma~\ref{lem:C1}. The modified clusters of the closure are included in the dotted DYD.}
\label{fig:closure}
\end{subfigure}
\caption{Illustrations of Corollary~\ref{cor:crumbs:growth}, Lemma~\ref{lem:C1} and Proposition~\ref{prop:closure}.}
\end{center}
\end{figure}

We next transform these results for infinite regions into a result for droplets. It states that a crumb next to a droplet cannot grow significantly (see Figure~\ref{fig:closure}).
\begin{lem}
\label{lem:C1}
Let $C_1$ be sufficiently large depending on $\cU$ and $\cS$. Let $D$ be a DYD at distance at least $C_3$ from $\partial$ or be a CDYD and let $\kappa$ be a crumb. Then $[\kappa]_{D\cup\partial}=[\kappa]_D$ is within distance $C_1$ of $\kappa$.
\end{lem}
\begin{proof}
Assume that $D$ is a DYD at distance at least $C_3$ from $\partial$. The proof of~\cite[Lemma~6.10]{Bollobas14} applies using \eqref{eq:def:DYD}, Observation~\ref{obs:C1:well:defined}, Corollary~\ref{cor:C1:well:defined} and the arguments in the proof of Corollary \ref{cor:crumbs:growth} to give the result for $[\kappa]_D$, which is therefore at distance at least $C_2-C_1$ from $\partial$ since $d(\kappa,\partial)\ge C_2$, so that in fact $[\kappa]_D=[\kappa]_{D\cup\partial}$.

Assume next that $D$ is a CDYD. Then actually $D\cup\partial$ can be viewed as a DYD on the entire plane without boundary specified by an infinite number of vectors $x_i$, so that we are in the previous case. In order to avoid introducing the corresponding notion of infinite DYD, one can consider an increasing exhaustive sequence of DYD $D_i$ converging to $D\cup\partial$ in the product topology and apply the previous result for $[\kappa]_{D_i}$, which will thereby apply to $D\cup\partial$. Finally, $[\kappa]_D=[\kappa]_{D\cup\partial}$ follows, since $d([\kappa]_{D\cup\partial},\partial)\ge C_2-C_1$.
\end{proof}

The next proposition is key to making the output of the algorithm essentially invariant under the KCM dynamics without having to pay for the fact that the closure for the bootstrap percolation dynamics of infections at equilibrium is not at all at equilibrium itself.
The proof is illustrated in Figure~\ref{fig:closure}.
\begin{prop}[Closure]
\label{prop:closure}
Let $K$ be a finite set and $\cD'$ be the collection of droplets given by the modified droplet algorithm with input $[K]_\partial$. Let $\cD$ be the output of the droplet algorithm for $K$. Then
\[\forall D'\in\cD'\,\exists D\in\cD,D'\subset D.\]
\end{prop}
\begin{proof}
Let $\cK$ be the set of crumbs for $K$. Set $\kappa_0=\bigcup_{\kappa\in\cK} \kappa$.\\

\par\textbf{Claim 1.} For each crumb $\kappa\in\cK$ its closure $[\kappa]=[\kappa]_\partial$ consists of at most $\alpha-1$ modified crumbs of $[\kappa]$ all contained within distance $C_1$ from $\kappa$. 
\begin{proof}[Proof of Claim 1]
There exists a set $P_\kappa$ as in Definition~\ref{def:crumb}, such that $[P_\kappa]\supset \kappa$ and thus $[P_\kappa]\supset[\kappa]$, which proves that all connected components of $[\kappa]$ for $\Gamma'$ are modified crumbs. The fact that $[\kappa]$ is within distance $C_1$ of $\kappa$ (and thus at distance at least $C_2'$ from $\partial$) was proved in Corollary~\ref{cor:crumbs:growth}, which also shows that $[\kappa]=[\kappa]_\partial$, since $\kappa$ is at distance more than $C_2$ from $\partial$.
\end{proof}
We can thus define $\cK'(\kappa)$ to be the set of modified crumbs of $[\kappa]_\partial$, so that their union is disjoint and equal to $[\kappa]_\partial$. Moreover, crumbs in $\cK$ are at distance at least $C_2$ from each other, so for any two of them $\kappa_1\neq \kappa_2$ we have that any $\kappa_1'\in\cK'(\kappa_1)$ and $\kappa_2'\in\cK'(\kappa_2)$ are at distance at least $C_2-2C_1\gg C_2'$ and also at such distance from $\partial$, so that $[\kappa_0]_\partial=\bigcup_{\kappa\in\cK} [\kappa]_\partial$ has no modified cluster and consists of modified crumbs at distance at most $C_1$ from $\kappa_0$.

For a droplet $D\in\cD$ consider the set of vectors $Y$ and $x$ ($x$ is absent for CDYD) defining it. Then define $\ring Y=Y+C_4u_0/C_1$ and $\ring x=x+C_4v_0/C_1$, where $u_0\in \bbR^2$ is the vector such that $\<u_0,u_1\>=\<u_0,u_2\>=-1$ and $v_0$ is defined identically in terms of $v_1$ and $v_2$. We denote by $\ring D$ the droplet defined by $\ring Y$ and $\ring x$ and call it a \emph{shrunken droplet}. Let $D_0=\bigcup_{D\in\cD} D$ and $\ring D_0=\bigcup_{D\in\cD}\ring D$. It is clear that $\ring D$ is at distance at least $C_4/C_1$ from $\Lambda\setminus D$ for all droplets $D$. In particular, all shrunken droplets are at distance at least $C_4/C_1$ from each other and shrunken DYD are at distance at least $C_4/C_1$ from $\partial$, so that Lemma~\ref{lem:C1} applies to them and $[\ring D_0]_\partial=\ring D_0$.\\

\par\textbf{Claim 2.}
$\ring D_0\cup \kappa_0\supset K$.
\begin{proof}[Proof of Claim 2]
Note that it is enough to prove that the clusters of $K$ are contained in $\ring D_0$. Assume that there exists $a\in K\setminus \ring D_0$ and $a\in C$ for some cluster. Then, $Q(C) \cap \Lambda$ is contained in some $D \in \mathcal{D}$, which is defined by $Y$ and $x$ ($x$ is absent for CDYD). Then since $a \not \in \ring D$, either for all $\ring y_i\in \ring Y$ we have $a\not\in \bbH_{u_1}(\ring y_i)\cap\bbH_{u_2}(\ring y_i)$ or $a\not\in\bbH_{v_1}(\ring x)\cap\bbH_{v_2}(\ring x)$. In the former case, $a-C_4 u_0/C_1\not\in \bbH_{u_1}(y_i)\cap\bbH_{u_2}(y_i)$ for all $y_i\in Y$. However, $Q(C)$ contains the ball of radius $C_4$ centered at $a$ and $\|u_0\|=O(1)$, so we get a contradiction. If $a\not\in\bbH_{v_1}(\ring x)\cap\bbH_{v_2}(\ring x)$, the first point on the segment from $a$ to $a-C_4 v_0/C_1$ that is not in $D$ is in $\Lambda$ and in $Q(C)$, hence a contradiction.
\end{proof}

\par\textbf{Claim 3.}
The set $[K]_\partial\setminus [\kappa_0]_\partial$ is within distance $C_3$ of $\ring D_0$.
\begin{proof}[Proof of Claim 3]
By Claim~2 we have $K_0=\ring D_0\cup \kappa_0\supset K$. It then clearly suffices to prove that $[K_0]_\partial\setminus[\kappa_0]_\partial$ is within distance $C_3$ of $\ring D_0$.

Consider a crumb $\kappa\in\cK$ at distance at most $C_2$ from $\ring D_0$, so at distance at most $C_2$ from a shrunken droplet $\ring D$ and necessarily at distance at least $C_4/C_1-C_2-C_3$ from any other shrunken droplet and from $\partial$ if $D$ is a DYD. By Lemma~\ref{lem:C1} $[\kappa]_{\ring D}=[\kappa]_{\ring D\cup\partial}$ is within distance $C_1$ of $\kappa$. Hence,
\begin{equation}
\label{eq:[K_0]:decomposition}
[K_0\cup\partial]=\ring D_0\cup\partial\cup[\kappa_0]\cup\bigcup_{\kappa,D}[\kappa]_{\ring D},
\end{equation}
where the last union is on couples $(\kappa,D)$ as above. Indeed, all $[\kappa]_{\ring D}$ and $[\kappa]$ (for different $\kappa$) are at distance at least $C_2-2C_1$ from each other and from $\ring D_0\setminus \ring D$ (by the reasoning above), so for each site of $\Lambda$ the intersection of the ball of radius $O(1)$ centered at it with the set on the right-hand side of~\eqref{eq:[K_0]:decomposition} coincides with the intersection with one of the sets $[\kappa\cup \ring D]$, $[\kappa]$ or $\ring D_0\cup\partial$, which are all stable, so no infections occur, which proves~\eqref{eq:[K_0]:decomposition}.

The claim follows easily from~\eqref{eq:[K_0]:decomposition}, since for every couple $\kappa,D$ the set $[\kappa]_{\ring D}$ is within distance $C_1$ of $\kappa$, which is itself at distance at most $C_2$ from $\ring D_0$, and $\kappa$ has diameter much smaller than $C_3$ by Corollary~\ref{cor:crumbs:growth}.
\end{proof}

Let $C'$ be a modified cluster of $[K]_\partial$ and assume for a contradiction that $C'\subset[\kappa_0]_\partial$. From Definition~\ref{def:cluster} we get that $C'$ is also a modified cluster of $[\kappa_0]_\partial$, but this is a contradiction, since $[\kappa_0]_\partial$ only consists of modified crumbs.

Since any modified cluster $C'$ of $[K]_\partial$ has diameter at most $C_3$ (by Definition~\ref{def:cluster}) and intersects $[K]_\partial\setminus[\kappa_0]_\partial$, which is within distance $C_3$ of $\ring D_0$ by Claim 3, we get that $C'$ is within distance $2C_3$ of $\ring D_0$. Therefore, $\bigcup_{C'\in\cC'([K]_\partial)} Q'(C')\subset D_0\cup\partial$, where the union is over all modified clusters of $[K]_\partial$, since $\diam (Q'(C'))\ll C_4/C_1\le d(\ring D_0,\Lambda\setminus D_0)$. As $\cD$ is the output of the droplet algorithm, $D_0$ is the union of disjoint DYD non-intersecting $\partial$ and CDYD, so it necessarily contains $\bigcup_{D'\in\cD'}D'$ (see Remark~\ref{rem:algo}), which concludes the proof.
\end{proof}

\begin{rem}
\label{rem:generalisations:algo}
It should be noted that the algorithm is more easily and naturally defined with no boundary, but that will not be sufficient for our purposes. However, this `free' algorithm is trivially obtained as a specialisation of ours. It is also possible to deal with more general boundaries, with infinite input sets, as well as with droplets defined by more directions and possibly with several rugged sides.
\end{rem}

\section{Renormalised East dynamics}
\label{sec:East}

In this section we map the original dynamics into an East one and conclude the proof of our main result. In Section~\ref{sec:geo} we introduce the necessary notation for the relevant geometry. In Section~\ref{sec:arrows} we consider a renormalised dynamics on the slices of Figure~\ref{fig:domain} by algorithmically selecting certain modified  spanned droplets of size $\Omega(1/q^\alpha)$. In Section~\ref{sec:eta} we further renormalise to recover an exact East dynamics where $q$ is replaced by $\qe$ corresponding to the probability of spanning such a droplet. Finally, in Section~\ref{sec:proof} we prove Theorem~\ref{th:main} roughly as in~\cite{Mareche20Duarte}.

\subsection{Geometric setup}
\label{sec:geo}
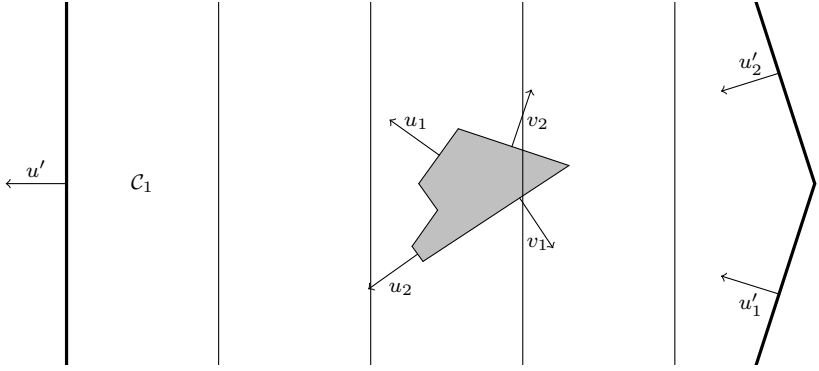
\begin{figure}
\floatbox[{\capbeside\thisfloatsetup{capbesideposition={right,top}}}]{figure}[\FBwidth]{
\begin{tikzpicture}[line cap=round,line join=round,x=0.8cm,y=0.8cm]
\clip(-7,-3) rectangle (6.4,3);
\fill[fill=black,fill opacity=0.25] (2.26,0.3) -- (0.44,0.91) -- (-0.21,0) -- (0.1,-0.44) -- (-0.32,-1.04) -- (-0.14,-1.29) -- cycle;
\fill[line width=0pt,color=white,fill=white,fill opacity=1.0] (-6,38.38) -- (-7,5) -- (-7,-5) -- (-6,-38.38) -- cycle;
\fill[line width=0pt,color=white,fill=white,fill opacity=1.0] (14.75,0.06) -- (-6,38.38) -- (6.3,0) -- cycle;
\fill[line width=0pt,color=white,fill=white,fill opacity=1.0] (14.75,0.06) -- (-6,-38.38) -- (6.3,0) -- cycle;
\draw (2.26,0.3)-- (0.44,0.91);
\draw (0.44,0.91)-- (-0.21,0);
\draw (-0.21,0)-- (0.1,-0.44);
\draw (0.1,-0.44)-- (-0.32,-1.04);
\draw (-0.32,-1.04)-- (-0.14,-1.29);
\draw (-0.14,-1.29)-- (2.26,0.3);
\draw [->] (0.13,0.47) -- (-0.69,1.06);
\draw [->] (-0.23,-1.17) -- (-1.04,-1.75);
\draw [->] (1.45,-0.24) -- (2,-1.07);
\draw [->] (1.32,0.61) -- (1.64,1.56);
\draw [->] (-6,0) -- (-7,0);
\draw [->] (5.71,1.83) -- (4.76,1.53);
\draw [->] (5.71,-1.83) -- (4.76,-1.53);
\draw[very thick] (-6,38.38)-- (-6,-38.38);
\draw[very thick] (-6,38.38)-- (6.3,0);
\draw[very thick] (-6,-38.38)-- (6.3,0);
\draw (-3.5,38.38)-- (-3.5,-38.38);
\draw (-1,38.38)-- (-1,-38.38);
\draw (1.5,38.38)-- (1.5,-38.38);
\draw (4,38.38)-- (4,-38.38);
\begin{scriptsize}
\draw[color=black] (-0.25,1) node {$u_1$};
\draw[color=black] (-0.5,-1.75) node {$u_2$};
\draw[color=black] (1.75,-1) node {$v_1$};
\draw[color=black] (1.75,1) node {$v_2$};
\draw[color=black] (-6.5,0.25) node {$u'$};
\draw[color=black] (5.25,2) node {$u'_2$};
\draw[color=black] (5.25,-2) node {$u'_1$};
\draw[color=black] (-4.75,0) node {$\cC_1$};
\end{scriptsize}
\end{tikzpicture}}
{\caption{The domain $V$ is the thickened triangle, a portion of which is displayed. Solid lines separate columns $\cC_i$. Inside the domain is drawn a DYD, which witnesses $\Phi(\omega)_3=\uparrow$.}
\label{fig:domain}}
\end{figure}

Let us start by defining the domain $V$ we will work in, recalling the notation from Lemma~\ref{lem:directions}. Roughly speaking, $V$ is an isosceles triangle with height $e^{1/(C_5q^\alpha)}$ directed by $u'$ (see Figure~\ref{fig:domain}). It is divided into `columns' $\cC_i$ perpendicular to $u'$ of width roughly $1/q^\alpha$, so that the origin of $\bbZ^2$ is in the middle of the last column, close to the tip of $V$.

More formally, set $L=1/(C_5q^{\alpha})$ and let $\iota$ be the smallest $x\ge1$ such that the site $\frac{x}{2q^\alpha}u'$ is in $\bbZ^2$, so that $\iota=1+O(q^\alpha)$. This way our columns will have width $\iota/q^{\alpha}$ and be separated along rational lines. We define the domain
\[V= \bbH_{u'}(e^Lu')\setminus \left(\bbH_{u_2'}(-\iota/(2q^\alpha)u')\cup\bbH_{u'_1}(-\iota/(2q^\alpha)u')\right).\]
Let us choose $C_5$ so that half the number of columns
\[N=e^Lq^\alpha/(2\iota)+1/4=e^Lq^\alpha(1/2+O(q^\alpha))\]
is an integer. We then partition the domain $V=\bigcup_{i=1}^{2N}\cC_i$ into columns with
\[\cC_{i}=\{x\in V \colon e^L-\iota(i-1)/q^\alpha > \<x,u'\>\ge e^L-\iota i/q^\alpha\},\]
so that $0$ is in the middle of $\cC_{2N}$ and $e^L u'\in\bbZ^2$. We shall refer to $\cC_i$ as the \emph{$i$-th column}. Finally, define the half-plane containing $\cC_{i+1}$, but not intersecting $\cC_{i}$
\[\bbH_i=\bbH_{u'}((e^L-\iota i/q^\alpha)u')\]
and the natural boundary for $\cC_i$
\[\partial_i=\bbH_i\cup \bar\partial,\]
obtained by considering $\cC_{j},j \geq i+1$ as fully infected, where
\[\bar\partial=\bbH_{u_2'}(-\iota/(2q^\alpha)u')\cup\bbH_{u'_1}(-\iota/(2q^\alpha)u').\]
Note that these boundaries
 are of the form considered in Section~\ref{sec:droplets}.

\subsection{Arrow variables}
\label{sec:arrows}
Let $\omega\in\Omega$. We will now define a collection of arrow variables which depend only on the restriction of $\omega$ to $V$. We naturally identify the restriction of $\omega$ to $V$ with the subset of $V$ where $\omega$ is $0$ and we use the notation $\omega=\varnothing$ to indicate that all sites are filled (healthy) in $V$, namely $\omega_x=1$ for all $x\in V$.  Let $\omega^{(0)}=\omega\cap V$. We define the position of the first \emph{up-arrow} as the smallest index $i_1(\omega)\in\{1,2,\dots,2N\}$ such that there is a modified spanned droplet of size at least $L$ for $[\omega^{(0)}]_{\partial_{i_1(\omega)}}$ with boundary $\partial_{i_1(\omega)}$. If no such $i_1$ exists, we say that there are no up-arrows and set $i_1(\omega)=\infty$. We further denote $\omega^{(1)}=\omega^{(0)}\cap\bbH_{i_1(\omega)}$ as soon as $i_1(\omega)<\infty$, while otherwise $\omega^{(1)}=\varnothing$.

We define the set   $I(\omega)=\{i_1(\omega),i_2(\omega),\dots\}\subset\{1,\dots,2N\}$
containing the positions of up-arrows recursively as follows. If there are no up-arrows, then $I=\varnothing$. Otherwise, we set $I(\omega)=\{i_1(\omega)\}\cup I(\omega^{(1)})$ and $\omega^{(k)}=(\omega^{(k-1)})^{(1)}$, which defines $\omega^{(k)}$ for all $k$. Let us note that if $i_1(\omega)\neq\infty$, then $i_1(\omega)<i_1(\omega^{(1)})$, since by definition $[\omega^{(1)}]_{\partial_{i_1(\omega)}}=\varnothing$.
Finally, we define $\Phi(\omega)\in\{\uparrow,\downarrow\}^{\{1,\dots,2N\}}$ 
as
\begin{equation*}
\Phi(\omega)_k = 
\begin{cases}
\uparrow &\text{ if }k\in I(\omega),\\
\downarrow &\text{ otherwise.}  
\end{cases}
\end{equation*}

The next Lemma states that the probability to find at least one up-arrow decays as
\[\qe= e^{-L}.\]
\begin{lem}\label{lem:no:active:column}
\[\mu(i_1<\infty)\le\qe.\]
\end{lem}
\begin{proof}
Fix $1\le i\le 2N$ and consider the event $i_1=i$. It is clearly included in the event $E_i$ that there is a modified spanned droplet of size at least $L$ for $[\omega^{(0)}]_{\partial_i}$ with boundary $\partial_i$. By Proposition~\ref{prop:closure} there is also a spanned droplet of size at least $L/C_1$ for $\omega^{(0)} \setminus \partial_i$ with boundary $\partial_i$. By Lemma~\ref{lem:AL} this implies that there is also a spanned connected droplet of size between $L/C_1^2$ and $2L/C_1^2$.
Then one can rewrite $E_i$ as the union over all such droplets $D$ of the event that $D$ is spanned. Note that for each discretised DYD $D\cap\bbZ^2$ the event that there exists a spanned DYD $D'$ with $D'\cap\bbZ^2=D\cap\bbZ^2$ coincides with the event that a suitably chosen such $D_0'$ is spanned. Indeed, the intersection of two DYD is a DYD by~\eqref{eq:def:DYD} and the spanning of all $D'$ depend only on the finite number of sites in $D\cap\bbZ^2$, so there is a finite number of possible events associated to different $D'$ and one can consider the intersection of a $D'$ defining each of these events. The same reasoning holds for CDYD and so for each discretised droplet $D\cap\bbZ^2$ one can bound the probability that there exists a spanned droplet with such discretisation using Lemma~\ref{lem:exp:decay:diam}. Thus, by the union bound on discretised droplets counted in Observation~\ref{obs:entropy}, one obtains
\[\mu(E_i)\le |V|.e^{L}2e^{-C_4L/C_1^2}\le\qe/(2N).\qedhere\]
\end{proof}

We next consider the event of having at least $n$ up-arrows
\[\mathcal{B}(n)=\{\omega\in\Omega: |I(\omega)|\geq n\}.\]
\begin{cor}\label{cor:Bn}
For any $1\le n\le 2N$ we have
\[\mu(\mathcal{B}(n))\le q_{\mathrm{eff}}^{n}.\]
\end{cor}
\begin{proof}
We prove the statement by induction on $n$. The base, $n=1$, is given by Lemma~\ref{lem:no:active:column}. For $n>1$ we have
\begin{align*}
\mu(|I|\ge n)=&\sum_{i=1}^{2N}\mu(i_1(\omega)=i;|I(\omega\cap\bbH_{i})|\ge n-1)\\
\le&\sum_{i=1}^{2N}\mu(i_1=i)\mu(|I|\ge n-1)\\
\le&q_{\mathrm{eff}}^{n},
\end{align*}
where we used that the event $i_1=i$ only depends on $\omega\setminus\bbH_{i}$ ($i_1$ is a stopping time for the filtration induced by the columns) and that the event $|I|\ge n-1$ is increasing for the order defined by $\omega \preceq \omega'$ when $\omega \subset \omega'$.
\end{proof}

We will now state a key deterministic property of the arrows under legal moves of the KCM dynamics. 

\begin{lem}\label{lem:chain}
Let $\omega\in\Omega$.  Let $x\in \cC_i$ be such that $\omega_x=1$
and the constraint at $x$ is satisfied by $\omega\cup\bar\partial$. Assume that $\Phi(\omega)\neq\Phi(\omega^x)$. Let $j=\max\{k \colon \Phi(\omega)_k\neq\Phi(\omega^x)_k\}$. Then 
\[\Phi(\omega)_{[i-1,j]}=(\uparrow,\downarrow,\uparrow,\downarrow,\uparrow,\dots)\text{, }\Phi(\omega^x)_{[i-1,j]}=(\uparrow,\uparrow,\downarrow,\uparrow,\downarrow,\dots)\text{ and }\Phi(\omega)_{[0,i-1]}=\Phi(\omega^x)_{[0,i-1]}\]
with the convention that $\Phi(\omega)_0=\uparrow$ for all $\omega$.
\end{lem}
\begin{proof}
We denote $\Phi:=\Phi(\omega)$ and $\Phi':=\Phi(\omega^x)$. Clearly, $\Phi_{[0,i-1]}=\Phi'_{[0,i-1]}$, since those values do not depend on $\omega\cap \bbH_{i-1}$. 

\par\textbf{Claim 1.}
Let $k\ge i$. If $\Phi_k=\uparrow$, then $\Phi_{[k+1,2N]}\ge \Phi'_{[k+1,2N]}$ for the lexicographic order associated to $\uparrow<\downarrow$. If $\Phi'_k=\uparrow$, then $\Phi_{[k+1,2N]}\le \Phi'_{[k+1,2N]}$.
\begin{proof}[Proof of Claim 1]
The two assertions being analogous, we only prove the first one, so assume that $\Phi_k=\uparrow$. Let $j'=\min\{l>k\colon\Phi_{l}=\uparrow\}$. Then there is a modified spanned droplet of size at least $L$ for $[\omega^{(0)}\cap\bbH_{k}]_{\partial_{j'}}$ with boundary $\partial_{j'}$. But this is also true for $\omega^x$ instead of $\omega$, as they coincide in $\bbH_{k}$, and in particular the position of the first up-arrow of $\Phi'$ after $k$ is at most~$j'$.
\end{proof}

\par\textbf{Claim 2.}
Let $k\ge i-1$ be such that $\Phi_{k}=\Phi'_k=\downarrow$. Then $k>j$ i.e. $\Phi_{[k,2N]}=\Phi'_{[k,2N]}$.
\begin{proof}[Proof of Claim 2]
We can clearly assume that $k < 2N$. Further assume for a contradiction that $\Phi_{k+1}=\uparrow$ and $\Phi'_{k+1}=\downarrow$. Let $i'=\max\{l<k \colon\Phi_l=\uparrow\}$. Then there exists a modified spanned droplet $D$ of size at least $L$ for $[\omega^{(0)}\cap\bbH_{i'}]_{\partial_{k+1}}$ with boundary $\partial_{k+1}$. By Lemma~\ref{lem:AL} we can assume that $L\le |D|\le C_1L$. However, if $d(D,\cC_{k+1})>C_5$, then $D$ is also modified spanned for $[\omega^{(0)}\cap\bbH_{i'}]_{\partial_{k}}$ with boundary $\partial_{k}$, contradicting the definition of $i'$. Indeed, from the output of the modified droplet algorithm for $[\omega^{(0)}\cap\bbH_{i'}]_{\partial_{k}}\cap D$ with boundary $\partial_k$ we can create a collection $\hat{\mathcal{D}}$ of droplets for $\partial_{k+1}$ by extending CDYD appropriately, thus $\hat{\mathcal{D}}$ contains $Q'(C')\setminus\partial_k=Q'(C')\setminus\partial_{k+1}$ for every modified cluster $C'$ of $[\omega^{(0)}\cap\bbH_{i'}]_{\partial_{k}}\cap D$ with boundary $\partial_k$. Moreover, the modified clusters of  $[\omega^{(0)}\cap\bbH_{i'}]_{\partial_{k+1}}\cap D$ with boundary $\partial_{k+1}$ are contained in the modified clusters of $[\omega^{(0)}\cap\bbH_{i'}]_{\partial_{k}}\cap D$ with boundary $\partial_k$, so $\hat{\mathcal{D}}$ contains the output of the modified droplet algorithm for $[\omega^{(0)}\cap\bbH_{i'}]_{\partial_{k+1}}\cap D$ with boundary $\partial_{k+1}$ by Remark~\ref{rem:algo}, itself containing $D$.

Therefore, $d(D,\cC_{k+1})\le C_5$. Moreover, $D$ is not modified spanned for $[(\omega^x)^{(0)}\cap\bbH_{k-1}]_{\partial_{k+1}}$ with boundary $\partial_{k+1}$ (otherwise $\Phi'_{[k,k+1]}\neq(\downarrow,\downarrow)$). Therefore, there exists a site $y\in D$ such that
\[y\in[\omega^{(0)}\cap\bbH_{i'}]_{\partial_{k+1}}\setminus[(\omega^x)^{(0)}\cap\bbH_{k-1}]_{\partial_{k+1}}.\]

We consider two subcases. First assume that $d(x,\bbR^2\setminus\bbH_{i-1})\ge C_1$. Then, the constraint at $x$ is satisfied by $(\omega\cap\bbH_{i-1})\cup\bar\partial$, so $[\omega^{(0)}\cap\bbH_{k-1}]_{\partial_{k+1}}=[(\omega^x)^{(0)}\cap\bbH_{k-1}]_{\partial_{k+1}}$, and there is a path
\[P\subset[\omega^{(0)}\cap\bbH_{i'}]_{\partial_{k+1}}\setminus[(\omega^x)^{(0)}\cap\bbH_{k-1}]_{\partial_{k+1}}\]
from $\bbR^2\setminus\bbH_{k-1}$ to $y$ such that each two consecutive sites are at distance at most $O(1)$. But $d(y,\bbR^2\setminus\bbH_{k-1})\ge \iota/q^{\alpha}-\diam(D)-C_5\ge C_2(L+1)$, so one can find a subpath $P'\subset\cC_k\cap P$ of diameter at least $C_2L$. Yet, it is clear that $P'\subset[\omega^{(0)}\cap\bbH_{i'}]_{\partial_{k}}$ implies the existence of a modified spanned droplet of size larger than $L$ with boundary $\partial_k$, so one would have an up-arrow of $\Phi$ in $[i'+1,k]$ -- a contradiction. If, on the contrary, $d(x,\bbR^2\setminus\bbH_{i-1})\le C_1$, we can redo the same reasoning, but $P$ needs to extend to either $\bbR^2\setminus\bbH_{k-1}$ or $x$, both of which are sufficiently far from $y$.

Thus, $\Phi_{k+1}=\Phi'_{k+1}$, as the case $\Phi_{k+1}=\downarrow,\Phi'_{k+1}=\uparrow$ is treated identically. But then either both are $\uparrow$, in which case we are done by Claim~1 or both are $\downarrow$ and we are done by induction.
\end{proof}
It is easy to see that the only non-identical arrow sequences $\Phi_{[i-1,j]}$ and $\Phi_{[i-1,j]}'$ satisfying the two claims are $(\uparrow,\downarrow,\uparrow,\downarrow,\dots)$ and $(\uparrow,\uparrow,\downarrow,\uparrow,\dots)$ (in this order using that $\omega_x=1$). Indeed, by Claims~1 and~2 $\Phi_k\neq\Phi'_k$ for all $i\le k\le j$, by Claim~1 one cannot have two consecutive up arrows neither in $\Phi$ nor in $\Phi'$ in the interval $[i,j]$ and by Claim~2 $\Phi_{i-1}=\Phi'_{i-1}=\uparrow$.
\end{proof}

\subsection{Renormalised East dynamics}
\label{sec:eta}

We partition $\{1,\dots,2N\}$ into blocks $B_i=\{2i-1,2i\}$ for $1\le i\le N$. Given $\omega\in\Omega$, we define $\eta(\omega)\in\{0,1\}^{\{1,\dots,N\}}$ by
\[\eta(\omega)_i=\1_{\{\forall j\in B_i:\Phi(\omega)_j=\downarrow\}}
\]
for all $i\in \{1,\dots N\}$. Let
\[n=\lfloor L\rfloor=\left\lfloor\frac{1}{C_5q^\alpha}\right\rfloor<\lfloor\log_2 N\rfloor.\]

Recall the definition  of legal paths, Definition \ref{def:legal}. Given an event $\mathcal E\subset\Omega$ and a legal path $\gamma=(\omega_{(0)},\dots,\omega_{(k)})$ we will say that $\gamma\cap\mathcal E=\varnothing$ if
$\omega_{(i)}\not\in \mathcal E$ for all $i\in\{0,\dots,k\}$. 
Also, given $\omega\in\Omega$ and $\mathcal A\subset\Omega$, we say that $\gamma$ connects $\omega$ to $\mathcal A$ if $\omega_{(0)}=\omega$ and $\omega_{(k)}\in\mathcal A$. 
Recall that $\mathcal B(n)\subset \Omega$ is the set of configurations with at least $n$ up-arrows.  The following is a straightforward but important corollary of Lemma~\ref{lem:chain}.
\begin{cor}\label{cor:etaEast}
For any legal path $(\omega_{(0)},\dots,\omega_{(k)})$, the path $(\eta(\omega_{(0)}),\dots,\eta(\omega_{(k)}))$ is legal for the East model on $\{1,\dots,N\}$ defined by fixing $\eta_0=0$.
\end{cor}
\begin{proof}
By Lemma~\ref{lem:chain} $\eta(\omega_{(j)})\neq \eta(\omega_{(j+1)})$ implies that $\Phi(\omega_{(j)})$ and $\Phi(\omega_{(j+1)})$ only differ on an alternating chain of arrows ending in some $B_i$, preceded by $\uparrow$. Then clearly $\eta(\omega_{(j)})_l=\eta(\omega_{(j+1)})_l$ for all $l\neq i$ and $\eta(\omega_{(j)})_{i-1}=0$.
\end{proof}

Let $\Omega_{\downarrow}$ and $\Omega_{\uparrow}^{2N}$ 
be respectively the set of configurations which do not have up-arrows, and the set of configurations with an up-arrow in the $2N$-th  column, namely
\begin{align*}
\Omega_{\downarrow}&{}=\{\omega\in\Omega:\Phi(\omega)=(\downarrow,\dots,\downarrow)\},\\
\Omega_{\uparrow}^{2N}&{}=\{\omega\in\Omega:\Phi(\omega)_{2N}=\uparrow\}.
\end{align*}

Combining the last corollary with Proposition~\ref{lem:comb_East}, we obtain the most important input for the proof of the main result.
\begin{cor} \label{keycor}
For any $\omega\in\Omega_{\downarrow}$
there does not exist a legal path $\gamma$ with $\gamma\cap \mathcal B(n+1)=\varnothing$ 
connecting $\omega$ to $\Omega_{\uparrow}^{2N}$.
\end{cor}

\subsection{Proof of Theorem \ref{th:main}}
\label{sec:proof}
To prove Theorem \ref{th:main} it is sufficient to prove the lower bound for the mean infection time and use the following inequality (see  \cite[Theorem 4.4]{Cancrini09} and also \cite[Section~2.2]{Martinelli19b})
\begin{equation}
\label{eq:trel:tau}
\trel\geq q \bbE(\tau_0).
\end{equation}
However, it is instructive  to construct at this stage a test function that directly gives the desired lower bound on $\trel$ without going through the comparison with the mean infection time. Indeed, the mechanism will appear more clearly this way. 

\paragraph{Proof of Theorem~\ref{th:main} for $\trel$}
We define the event
\[\tilde \cA=\{\omega\in\Omega\colon \exists 
{\mbox{ a legal path $\gamma$ with $\gamma\cap \mathcal B(n)=\varnothing$ connecting $\omega \cup (\bbZ^2\setminus V)$ to $\Omega_{\downarrow}$}}\}\]
and the test function $f:\Omega\to\{0,1\}$
\[f=\1_{\tilde \cA}.\]
Then, by Definition~\ref{def:PC} we get
\begin{equation}
    \label{eq:lb_trel}
    \trel\geq \frac{\mu(\tilde{\cA})(1-\mu(\tilde{\cA}))}{\cD(f)},
\end{equation}
where the Dirichlet form $\cD(f)$ is defined in \eqref{def:Dirichlet}.

\begin{lem}[Bounds on $\mu(\tilde \cA)$]
\label{lem:muAtilde}
\[\mu(\tilde {\cA})\left(1- \mu(\tilde {\cA})\right)\geq \exp\left(\frac{\log q}{C_4q^{\alpha}}\right).\]
\end{lem}
\begin{proof}
By Lemma~\ref{lem:no:active:column} we have
\[\mu(\tilde \cA)\geq \mu(\Omega_{\downarrow})\geq 1-\qe\ge 1/2.\]
On the other hand,
\[1-\mu(\tilde{\cA}) \geq  
\mu(\Omega_{\uparrow}^{2N})\geq q^{C_1L}\ge\exp(C_1\log q/(C_5q^\alpha)),\]
where we used Corollary~\ref{keycor} for the first inequality as well as the fact that if $(\omega_{(0)},\dots,\omega_{(k)})$ is a legal path, then $(\omega_{(k)},\dots,\omega_{(0)})$ is one as well, and for the second inequality we notice that for the $2N$-th arrow to be up it is sufficient to have an empty segment of length $C_1L$ in $\mathcal{C}_{2N}$. 
\end{proof}

\begin{lem}[Estimate of the Dirichlet form]
\label{lem:diri}
$\mathcal{D}(f)\le \exp\left(-1/(C_5^3q^{2\alpha})\right)$.
\end{lem}
\begin{proof}
Using the fact that $f(\omega)$ depends only on the values of $\omega$ in $V$, we get
\begin{equation}
\label{eq:cD}\cD(f)=\sum_{x\in V} \mu(c_x {\mbox{Var}}_x(f))=q(1-q)\sum_{x\in V} \mu\left(c_x \1_{\{\omega\in \tilde{\cA},\,\omega^x\not\in \tilde{\cA}\}}+c_x\1_{\{\omega\not\in \tilde{\cA},\,\omega^x\in \tilde{\cA}\}}\right)\le |V|\mu(\cB(n-1)),
\end{equation}
since, by Lemma~\ref{lem:chain} $||I(\omega)|-|I(\omega^x)||\le 1$ when $c_x=1$, so the indicators both imply $\omega\in\cB(n-1)$. Indeed, $\omega\in \tilde{\cA}$ implies the existence of a legal path $\gamma$ from $\Omega_{\downarrow}$ to $\omega\cup (\bbZ^2\setminus V)$ with each configuration not in $\cB(n)$. Since $c_x=1$, the path $\bar\gamma$ obtained by adding the transition from $\omega\cup (\bbZ^2\setminus V)$ to $\omega^x\cup (\bbZ^2\setminus V)$ is also legal, thus the hypothesis $\omega^x\not\in \tilde{\cA}$ is not satisfied unless $\omega^x\in \cB(n)$ (and similarly for $\omega\not\in\tilde\cA,\omega^x\in\tilde\cA$). 
Thus, the result follows by using Corollary~\ref{cor:Bn}.
\end{proof}

Then the lower bound for $\trel$ of Theorem \ref{th:main} follows
from \eqref{eq:lb_trel}, Lemma \ref{lem:muAtilde} and Lemma \ref{lem:diri}.

The above proof, together with the matching upper bound of Theorem~2(a) of \cite{Martinelli19b} indicate that the bottleneck dominating the time scales is the creation of $\Theta(\log(1/\qe))$ simultaneous droplets of probability $\qe$.

\paragraph{Proof of Theorem~\ref{th:main} for $\bbE(\tau_0)$}
The proof of the lower bound for the infection time  follows a similar route, with some complications due to the fact that we have to identify a (sufficiently likely) initial set starting from which
we have to go through the bottleneck configurations before infecting the origin.

By \cite[Corollary~3.4]{Mareche20Duarte}, to prove the desired lower bound on $\bbE(\tau_0)$ 
it suffices to construct a local function $\phi=\phi_q$ such that
\begin{enumerate}[label=(\roman*)]
\item \label{cond:i} $\mu(\phi^2)=1$,
\item \label{cond:iii} $\frac{\mu(\phi)^4}{\mathcal{D}(\phi)}\ge\exp(1/(C_5^4q^{2\alpha}))$,
\item \label{cond:iv} $\phi(\omega)=0$ if $\omega_0=0$.
\end{enumerate}
Inspired by \cite{Mareche20Duarte}
we let 
\[\Omega_g=\Omega_{\downarrow}\cap\{\omega\in\Omega:\omega_{\Lambda_0}=1\}\]
 where $\Lambda_0=\{x\in \bbZ^2 \colon d(x,0)\le1/(4q^{\alpha})\}\subset\cC_{2N}$ and
\[\cA=\{\omega\in\Omega\colon\exists {\mbox{ a legal path $\gamma$ with $\gamma\cap \cB(n)=\varnothing$ connecting $\omega\cup (\bbZ^2\setminus V)$ to $\Omega_g$}}\}.\]
Then we set 
\begin{equation}
\label{test}
\phi(\cdot)=\1_{{\mathcal{A}}}(\cdot)/\mu({\mathcal{A}})^{1/2}.
\end{equation}
We are now left with proving that this function satisfies \ref{cond:i}-\ref{cond:iv} above.

Property~\ref{cond:i} follows immediately from \eqref{test}. In order to verify \ref{cond:iii} we start by establishing a lower bound on $\mu(\cA)$.
 By definition it holds that
\begin{equation}
\label{eq:muA}
\mu(\cA)\geq \mu(\Omega_g)\geq \mu(\omega_{\Lambda_0}=1) \mu(\Omega_{\downarrow})\geq e^{-O(1)/q^{2\alpha-1}}(1-\qe)= e^{-O(1)/q^{2\alpha-1}},
\end{equation}
where we used Harris' inequality~\cite{Harris60} ($\{\omega_{\Lambda_0}=1\}$ and $\Omega_{\downarrow}$ are increasing events if we consider that $\omega \leq \omega'$ when $\omega_x \leq \omega'_x$ for all $x \in \bbZ^2$), Lemma~\ref{lem:no:active:column} and
$|\Lambda_0|=O(1/q^{2\alpha})$.

Furthermore, one can repeat the proof of Lemma~\ref{lem:diri} to obtain
\begin{equation}
\label{Dir}
   \mathcal D(\phi)\leq e^{-1/(C_5^3q^{2\alpha})}.
\end{equation}
Thus, recalling~\eqref{eq:muA}, Property \ref{cond:iii} holds.

We are therefore only left with proving the next lemma establishing Property~\ref{cond:iv}, completing the proof of Theorem~\ref{th:main}.
\begin{lem}
\label{persistence}
Let $\omega$ be such that $\omega_0=0$. Then any legal path connecting $\Omega_g$ to $\omega$
intersects $\cB(n)$.
\end{lem}
As in the lower bound on $1-\mu(\tilde{\cA})$ for $\trel$, the proof relies on Corollary~\ref{keycor}, but an additional complication arises due to the fact that emptying the origin does not a priori require creating a critical droplet nearby.

\begin{proof}[Proof of Lemma \ref{persistence}]
Suppose for a contradiction that there exists a configuration $\omega$ with $\omega_0=0$, a configuration $\omega_{(0)}\in\Omega_g$ and a legal path $\gamma=(\omega_{(0)},\dots,\omega_{(k)})$ with $\omega_{(k)}=\omega$ and $\omega_{(j)}\not\in \mathcal B(n)$ for all $j\in \{0,\dots,k\}$. Assuming without loss of generality that $\omega_{(j)}\neq\omega_{(j-1)}$ for all $j$, let $x_j$ be such that $\omega_{(j)}=(\omega_{(j-1)})^{x_j}$. 
Consider the path $\tilde\gamma=(\tilde\omega_{(0)},\dots,\tilde\omega_{(k)})$ obtained by performing the same updates as for $\gamma$ except for flips in the column $\mathcal C_{2N}$, which are performed only if they correspond to emptying sites. More precisely, we let $\tilde\omega_{(0)}=\omega_{(0)}$ and
\[\tilde\omega_{(j)}=
\begin{cases}
(\tilde \omega_{(j-1)})^{x_j}&\text{ if }x_j\not\in\cC_{2N}\text{ or }(\tilde\omega_{(j-1)})_{x_j}=1,\\
\tilde \omega_{(j-1)}&\text{ otherwise.}
\end{cases}
\]
It is not difficult to verify by induction that $\tilde \gamma$
is also a legal path with $\tilde\omega_{(j)}\le\omega_{(j)}$ for all $j$ (where $\omega \leq \omega'$ when $\omega_x \leq \omega'_x$ for all $x \in \bbZ^2$) and that $\tilde\omega_{(j)}$ and $\omega_{(j)}$ coincide outside of $\mathcal{C}_{2N}$. Then $(\tilde\omega_{(k)})_0\le(\omega_{(k)})_0=0$ and by definition $(\tilde\omega_{(0)})_{\Lambda_0}=1$. Therefore, since inside $\cC_{2N}$ each site that has been emptied in $\gamma$ is also empty in $\tilde\omega_{(k)}$, we conclude that necessarily $\tilde\omega_{(k)}\cap\cC_{2N}$\ contains a (modified) spanned droplet of size $1/(4C_1q^{\alpha})>L$ with boundary $\partial_{2N}=\bar\partial$. Indeed, there is a path of sites $x$ with steps of size $O(1)$ from $\bbZ^2\setminus\Lambda_0$ to $0$ such that $(\tilde\omega_{(k)})_x=0$. This means that $\tilde\omega_{(k)} \in\Omega_{\uparrow}^{2N}$. Furthermore, for all $j$ we have $\Phi(\tilde\omega_{(j)})_{[1,2N-1]}=\Phi(\omega_{(j)})_{[1,2N-1]}$, as those do not depend on the sites in $\cC_{2N}$. Thus, using Corollary~\ref{keycor}, together with the facts that 
$\tilde \omega_{(0)}\in\Omega_g\subset\Omega_{\downarrow}$, $\tilde\omega_{(k)} \in\Omega_{\uparrow}^{2N}$ and $\tilde\gamma\cap \cB(n+1)=\varnothing$, we reach a contradiction.
\end{proof}

\section{Open problems}
\label{sec:open}
With Theorem~\ref{th:main} the scaling of the infection time is determined up to a polylogarithmic factor. The next natural question is to pursue determining this factor in the spirit of the refined universality result of~\cite{Bollobas14}. For the moment there is only one critical model with infinitely many stable directions for which this is known --- the Duarte model~\cite{Mareche20Duarte}. In that case the corrective factor is $\Theta((\log q)^4)$. However, for bootstrap percolation there are already two different possible behaviours of this factor depending on whether the model is balanced or unbalanced (see Definition~\ref{def:balanced}). Based on this one could expect the following.
\begin{conj}
\label{conj:logs}
Let $\cU$ be a critical update family with an infinite number of stable directions.
\begin{itemize}
\item If $\cU$ is balanced, then 
\[\bbE(\tau_0)=\exp\left(\frac{\Theta(1)}{q^{2\alpha}}\right).\]
\item If $\cU$ is unbalanced, then 
\[\bbE(\tau_0)=\exp\left(\frac{\Theta\left(\left(\log q\right)^4\right)}{q^{2\alpha}}\right).\]
\end{itemize}
The same asymptotics hold for $\trel$.
\end{conj}

In other words we expect the lower bound of Theorem~\ref{th:main} to be sharp for balanced models, while the upper bound of~\cite[Theorem~2(a)]{Martinelli19b} to be sharp for unbalanced ones. 
The balanced case is not hard and only requires an improvement of the approach of~\cite{Martinelli19b}. It will be treated in a future work, since it shares none of the techniques discussed here. In the unbalanced case the $(\log q)^4$ should arise as the square of the $(\log q)^2$ factor for bootstrap percolation, itself caused by the one-dimensional geometry and larger size of critical droplets. This is indeed what happens for the Duarte model \cite{Mareche20Duarte}, an example of unbalanced critical constraint. 

\section*{Acknowledgements}
This work has been supported by ERC Starting Grant 680275 MALIG and by ANR-15-CE40-0020-01. We wish to thank Fabio Martinelli for numerous and enlightening discussions and the Departement of Mathematics and Physics of University Roma Tre for its kind hospitality.
\bibliographystyle{plain}
\bibliography{Bib}

\begin{thebibliography}{10}

\bibitem{Aizenman88}
M.~Aizenman and J.~L. Lebowitz.
\newblock Metastability effects in bootstrap percolation.
\newblock {\em J. Phys. A}, 21(19):3801--3813, 1988.

\bibitem{Aldous02}
D.~Aldous and P.~Diaconis.
\newblock The asymmetric one-dimensional constrained {I}sing model: rigorous
  results.
\newblock {\em J. Stat. Phys.}, 107(5-6):945--975, 2002.

\bibitem{Bakry06}
D.~Bakry.
\newblock Functional inequalities for {M}arkov semigroups.
\newblock In {\em Probability measures on groups: recent directions and
  trends}, pages 91--147. Tata Inst. Fund. Res., Mumbai, 2006.

\bibitem{Balister16}
P.~Balister, B.~Bollob\'as, M.~Przykucki, and P.~Smith.
\newblock Subcritical {$\mathcal{U}$}-bootstrap percolation models have
  non-trivial phase transitions.
\newblock {\em Trans. Amer. Math. Soc.}, 368(10):7385--7411, 2016.

\bibitem{Berthier11}
L.~Berthier and G.~Biroli.
\newblock Theoretical perspective on the glass transition and amorphous
  materials.
\newblock {\em Rev. Mod. Phys.}, 83:587--645, 2011.

\bibitem{Bollobas17}
B.~Bollob\'as, H.~Duminil-Copin, R.~Morris, and P.~Smith.
\newblock The sharp threshold for the {D}uarte model.
\newblock {\em Ann. Probab.}, 45(6B):4222--4272, 2017.

\bibitem{Bollobas14}
B.~Bollob\'as, H.~Duminil-Copin, R.~Morris, and P.~Smith.
\newblock Universality of two-dimensional critical cellular automata.
\newblock {\em Proc. Lond. Math. Soc.}, to appear.

\bibitem{Bollobas15}
B.~Bollob\'as, P.~Smith, and A.~Uzzell.
\newblock Monotone cellular automata in a random environment.
\newblock {\em Combin. Probab. Comput.}, 24(4):687--722, 2015.

\bibitem{Cancrini08}
N.~Cancrini, F.~Martinelli, C.~Roberto, and C.~Toninelli.
\newblock Kinetically constrained spin models.
\newblock {\em Probab. Theory Related Fields}, 140(3-4):459--504, 2008.

\bibitem{Cancrini09}
N.~Cancrini, F.~Martinelli, C.~Roberto, and C.~Toninelli.
\newblock Facilitated spin models: recent and new results.
\newblock In {\em Methods of contemporary mathematical statistical physics},
  volume 1970 of {\em Lecture Notes in Math.}, pages 307--340. Springer,
  Berlin, 2009.

\bibitem{Cancrini10}
N.~Cancrini, F.~Martinelli, R.~Schonmann, and C.~Toninelli.
\newblock Facilitated oriented spin models: some non equilibrium results.
\newblock {\em J. Stat. Phys.}, 138(6):1109--1123, 2010.

\bibitem{Chleboun14}
P.~Chleboun, A.~Faggionato, and F.~Martinelli.
\newblock Time scale separation and dynamic heterogeneity in the low
  temperature {E}ast model.
\newblock {\em Comm. Math. Phys.}, 328(3):955--993, 2014.

\bibitem{Chleboun16}
P.~Chleboun, A.~Faggionato, and F.~Martinelli.
\newblock Relaxation to equilibrium of generalized {E}ast processes on
  {$\mathbb{Z}^d$}: renormalization group analysis and energy-entropy
  competition.
\newblock {\em Ann. Probab.}, 44(3):1817--1863, 2016.

\bibitem{Chung01}
F.~Chung, P.~Diaconis, and R.~Graham.
\newblock Combinatorics for the {E}ast model.
\newblock {\em Adv. in Appl. Math.}, 27(1):192--206, 2001.

\bibitem{Faggionato13}
A.~Faggionato, F.~Martinelli, C.~Roberto, and C.~Toninelli.
\newblock The {E}ast model: recent results and new progresses.
\newblock {\em Markov Process. Related Fields}, 19(3):407--452, 2013.

\bibitem{Fredrickson84}
G.~H. Fredrickson and H.~C. Andersen.
\newblock Kinetic {I}sing model of the glass transition.
\newblock {\em Phys. Rev. Lett.}, 53:1244--1247, 1984.

\bibitem{Garrahan11}
P.~Garrahan, P.~Sollich, and C.~Toninelli.
\newblock {Kinetically constrained models}.
\newblock In L.~Berthier, G.~Biroli, J-P. Bouchaud, L.~Cipelletti, and W.~van
  Saarloos, editors, {\em {Dynamical heterogeneities in Glasses, colloids and
  granular media and jamming transitions}}, International series of monographs
  on physics 150, pages 341--369. {Oxford University Press}, 2011.

\bibitem{Harris60}
T.~E. Harris.
\newblock A lower bound for the critical probability in a certain percolation
  process.
\newblock {\em Proc. Cambridge Philos. Soc.}, 56:13--20, 1960.

\bibitem{Hartarsky18subcritical}
I.~{Hartarsky}.
\newblock {$\mathcal{U}$-bootstrap percolation: critical probability,
  exponential decay and applications}.
\newblock {\em arXiv e-prints}, 2018.

\bibitem{Hartarsky19II}
I.~{Hartarsky}, F.~{Martinelli}, and C.~{Toninelli}.
\newblock {Universality for critical {{K}{C}{M}}: finite number of stable
  directions}.
\newblock {\em arXiv e-prints}, 2019.

\bibitem{Hartarsky18NP}
I.~Hartarsky and T.~Mezei.
\newblock Complexity of 2{D} bootstrap percolation difficulty: {A}lgorithm and
  {{N}{P}}-hardness.
\newblock {\em arXiv e-prints}, 2018.

\bibitem{Holroyd03}
A.~E. Holroyd.
\newblock Sharp metastability threshold for two-dimensional bootstrap
  percolation.
\newblock {\em Probab. Theory Related Fields}, 125(2):195--224, 2003.

\bibitem{Jackle91}
J.~J{\"a}ckle and S.~Eisinger.
\newblock A hierarchically constrained kinetic {I}sing model.
\newblock {\em Z. Phys. B Con. Mat.}, 84(1):115--124, 1991.

\bibitem{Liggett05}
T.~M. Liggett.
\newblock {\em Interacting particle systems}.
\newblock Classics in Mathematics. Springer-Verlag, Berlin, 2005.
\newblock Reprint of the 1985 original.

\bibitem{Mareche20Duarte}
L.~{Mar{\^e}ch{\'e}}, F.~{Martinelli}, and C.~{Toninelli}.
\newblock Exact asymptotics for {D}uarte and supercritical rooted kinetically
  constrained models.
\newblock {\em Ann. Probab.}, to appear.

\bibitem{Martinelli19b}
F.~{Martinelli}, R.~{Morris}, and C.~{Toninelli}.
\newblock {Universality results for kinetically constrained spin models in two
  dimensions}.
\newblock {\em Comm. Math. Phys.}, 369(2):761--809, 2019.

\bibitem{Martinelli19}
F.~Martinelli and C.~Toninelli.
\newblock Towards a universality picture for the relaxation to equilibrium of
  kinetically constrained models.
\newblock {\em Ann. Probab.}, 47(1):324--361, 2019.

\bibitem{Morris17}
R.~Morris.
\newblock Bootstrap percolation, and other automata.
\newblock {\em European J. Combin.}, 66:250--263, 2017.

\bibitem{Morris17b}
R.~Morris.
\newblock Monotone cellular automata.
\newblock In {\em Surveys in combinatorics 2017}, volume 440 of {\em London
  Math. Soc. Lecture Note Ser.}, pages 312--371. Cambridge Univ. Press,
  Cambridge, 2017.

\bibitem{Mountford95}
T.~S. Mountford.
\newblock Critical length for semi-oriented bootstrap percolation.
\newblock {\em Stochastic Process. Appl.}, 56(2):185--205, 1995.

\bibitem{Ritort03}
F.~Ritort and P.~Sollich.
\newblock Glassy dynamics of kinetically constrained models.
\newblock {\em Adv. Phys.}, 52(4):219--342, 2003.

\bibitem{Sollich99}
P.~Sollich and M.~R. Evans.
\newblock Glassy time-scale divergence and anomalous coarsening in a
  kinetically constrained spin chain.
\newblock {\em Phys. Rev. Lett.}, 83:3238--3241, 1999.

\end{thebibliography}
\end{document}